\documentclass{elsarticle}

\usepackage[ansinew]{inputenc}
\usepackage{graphics,graphicx}

\usepackage{makeidx}
\usepackage{epigraph}
\usepackage{yfonts}
\usepackage{wrapfig}
\usepackage{multicol}
\usepackage{setspace}
\usepackage{fancyhdr}
\usepackage{amsmath,amssymb}

\usepackage{enumerate}
\usepackage{longtable}
\usepackage{pifont}
\usepackage{hhline}
\usepackage{graphicx}
\usepackage{subfigure}
\usepackage[usenames]{color}
\usepackage{makeidx}
\usepackage{calrsfs}
\usepackage{setspace}
\usepackage{scalerel}

\newtheorem{theorem}{Theorem}
\newtheorem{lemma}{Lemma}
\newdefinition{remark}{Remark}
\newproof{proof}{Proof}
\newproof{pot}{Proof of Theorem \ref{thm2}}
\newtheorem{proposition}{Proposition}
\newtheorem{definition}{Definition}
\newtheorem{example}{Example}

\newcommand{\h}{\hspace*{0.5 cm}}

\newcommand{\N}{\mathbb{N}}

\newcommand{\Rset}{\ensuremath{\mathbb{R}}}

\newcommand{\Nset}{\ensuremath{\mathbb{N}}}

\journal{Journal of \LaTeX\ Templates}

\begin{document}
	
	\begin{frontmatter}
		
		\title{Schmidt Representation of Bilinear Operators on Hilbert Spaces}

		\author[rvt]{Eduardo Brandani da Silva\corref{cor1}}
		\ead{ebsilva@uem.br}
		\author[els]{Dicesar Lass Fernandez}
		\ead{dicesar@ime.unicamp.br}
		\author[els1]{Marcus Vinícius de Andrade Neves}
		\ead{marcusmatematico@hotmail.com}
		\cortext[cor1]{Corresponding author}
		
		\address[rvt]{Universidade Estadual de Maring\'a, Av. Colombo, 5790 - Zona 7, Maring\'a, Paran\'a Brazil}
		\address[els]{Imecc - Unicamp, Rua Sérgio Buarque de Holanda 651,  13083-859 -  Campinas - SP, Brazil}
		\address[els1]{Department of Mathematics, Federal University of Mato Grosso -
			Av. dos Estudantes 5055, 78735-901 Rondon\'opolis, MT, Brazil}

		\begin{abstract}
			 Current work defines Schmidt representation of a bilinear operator $T: H_1 \times H_2 \rightarrow K$, where $H_1, H_2$ and $K$ are separable Hilbert spaces. Introducing the concept of singular value and ordered singular value, we prove that if $T$ is compact, and its singular values are ordered, then $T$ has a Schmidt representation on real Hilbert spaces. We prove that the hypothesis of existence of ordered singular values is fundamental.
		\end{abstract}
		
		
	\end{frontmatter}
	

\section{Introduction}

Compact linear operators have a fundamental role in functional analysis and operator theory, and also have important applications, in particular in the study of boundary-value problems for elliptic differential equations. 

Bilinear operators are a fundamental research subject in harmonic analysis and functional analysis. Research on bilinear operators in the Hilbert spaces has also attracted the attention of researchers by a long time, see \cite{murr}. The study of Hilbert-Schmidt bilinear operators is also an important topic, see \cite{caran}, \cite{hern}, \cite{MM}. Also, in \cite{erdog}, convolutions of bilinear operators acting on a product of Hilbert spaces have been studied. Compactness of bilinear operators is also a very important research topic, and it is studied in \cite{AR}. Examples are also given in \cite{bern}.  

A fundamental result about linear operators on Hilbert spaces is the spectral theorem, which says that for a compact self-adjoint operator $L$ acting on a separable Hilbert space $H$, then there exists a system of orthonormal eigenvectors $x_1,x_2,x_3,\cdots$ of $L$ and corresponding eigenvalues $\lambda_1,\lambda_2,\lambda_3,\cdots$ such that
\[
L(x)=\displaystyle\sum^\infty_{n=1}
\lambda_n \langle x,x_n\rangle x_n \,,
\]
for all $x\in H$. The sequence $\{\lambda_n\}$ is decreasing and, if it is infinite, converges to $0$. The series on the right hand side converges in the operator norm of $\mathfrak{L}(H)$. In this case, this representation is called Schur Representation of $L$, see \cite[pag. 2]{JRR} for more details. Later, Erhard Schmidt showed that if $H_1$ and $H_2$ are Hilbert spaces and $L \in \mathfrak{L}(H_1,H_2)$ is a compact linear operator, then 
\[
L(x)=\displaystyle\sum^\infty_{n=1}
\lambda_n \langle x,x_n\rangle y_n \,,
\]
where the $\lambda_n$ are now the eigenvalues of the positive compact self-adjoint operator $S = (L^* L)^{1/2}$ acting on $H_1$, $S x_n = \lambda_n y_n$ and $y_n = \lambda_n^{-1} L x_n$, ($\lambda_n \neq 0$). The operator $S$ is called the absolute value of $L$ and denoted by $|L|$; its eigenvalues $\lambda_n$ are called the singular values of $L$.

Few works on the representation of bilinear and multilinear operators exist in the literature, among which we can mention \cite{amson1, amson2, EDM}. In \cite{EDM}, Shur representation of a compact bilinear operator $T : H \times H \to H$ was studied. Our main goal in current work is to introduce a Schmidt representation for bilinear operators. To attain it, we will define new concepts and prove some new results, showing similarities and differences with the linear case. For a compact bilinear operator $T : H_1 \times H_2 \to K$, we introduce the Schmidt representation
\[
T(x,y)=\sum^\infty_{n=1} \tau_n \langle x,x_n \rangle \langle y, y_n\rangle z_n \,,
\]
for all $(x,y) \in H_1 \times H_2$. We prove that if $T$ has such representation, then it is compact, and the $\tau_n$ are ordered singular values of $T$. On other hand, we will prove that if $T$ satisfy these conditions, then it has such representation.

The work is divided in the following sections. Making an analogy with the linear case, Sections $2$ introduces the notations and shows the main properties of bilinear operators. Section $3$ recall the Schmidt representation in the linear case. In Section $4$ we define singular value of a bilinear operator $T \in Bil(H_1 \times H_2,K)$. We prove that if $T$ is compact and nonzero, then $\|T\|$ is a singular value. In Section $5$ we define the Schmidt Representation for a bilinear operator $T : H_1 \times H_2 \to K$. We call the attention that even if a bilinear operator is compact, it may not have a Schmidt representation, giving a striking difference with the linear case. To get around this situation, we introduce the concept of ordered singular value of $T$, and we prove that if $T$ is compact, nonzero and all of its singular values are ordered, then $T$ has a Schmidt representation. In Section $6$ we study the connections between the Schmidt representation and the Schur representation of \cite{EDM}. Finally, in Section $7$, we prove that, if $T$ has a Schmidt representation, then it is a Hilbert-Schmidt operator. In all sections we give several examples and counter-examples of bilinear operators satisfying several conditions discussed in the text. Finally, we observe that our results will provide several examples of compact operators on real Hilbert spaces.

\section{Compact bilinear operators}

\h The standard notation from the Banach space
theory is used throughout the paper. Let $X$ and $Y$ be Banach spaces. We shall consider $X\times Y$
equiped with the norm
\[
||(x,y)||_{_{_{\infty}}} = \max
\{||x||_{_{X}},||y||_{_{Y}}\} .
\]
If we denote by $U_X$ the closed unity ball of a Banach space $X$,
we see that $U_{X\times Y} = \{(x,y) \in X \times Y ; \|(x,y)\|_{_{_{\infty}}} \leq 1\} = U_X\times U_Y$.

If $Z$ is a Banach space, we denote by $Bil(X \times Y,Z)$ the linear space of all bilinear operators $T$, from $X\times Y$ into $Z$, such that
\[
\|T\|=\|T\|_{Bil(X\times Y,Z)} = \sup \{\, ||T(x,y)||_Z;\;
\|(x,y)\|_\infty \leq 1\;\} < +\infty .
\]
Equiped with this norm, the linear space $Bil(X\times Y,Z)$ is a
Banach space. In particular, $Bil(X\times Y) = Bil(X\times Y, \Rset)$ is the
Banach space of all bounded bilinear forms. If $X = Y = Z$ we use only $Bil(X)$.

\begin{proposition} Let $T\in Bil(X\times Y,Z)$. Then,
	\[
	\|T\| = \displaystyle \sup_{\|(x,y)\|_{_{_{\infty}}} \leq 1 }
	\|T(x,y)\|_{_{\scaleto{Z}{5pt}}}= \sup_{\|x\|_{_{\scaleto{X}{5pt}}}
		= 1,\|y\|_{_{\scaleto{Y}{5pt}}} = 1} \|T(x,y)\|_{_{\scaleto{Z}{5pt}}} \,.
	\]
\end{proposition}
In what follows, let $B_{r,X}=\{x\in X : \|x\|_{_{\scaleto{X}{5pt}}} \leq  r \}$ be the closed ball of radius $r>0$ with center in the origin of a normed space $X$.\\

Given an operator  $T \in Bil(X \times Y,Z)$, in general, the set $T(X \times Y)$ is not a subspace of $Z$. It is called image of $T$. The rank of $T$ is defined by
\[
{\rm rank}(T) = {\rm dim}_Z([T(X \times Y)]) \,,
\]
where $[ S ]$ denotes the subspace generated by a set $S \subset X$.

\begin{definition}
	Let $ X $, $ Y $ and $ Z $ be normed vector spaces. A bilinear operator $ T: X \times Y \longrightarrow Z $ is called a compact bilinear operator if $ T (U_X \times U_Y) $ is pre-compact in $ Z $, that is, the closure $ \overline{T (U_X \times U_Y)} $ is compact in $Z$.
\end{definition}

\section{Singular Values and Schmidt Representation in the Linear Case}

We recall basic definitions on Schmidt representation for the linear case, see \cite{Pie} for more details. Let $H$ and $K$ be real Hilbert spaces and $T : H \rightarrow K$ a bounded linear operador. Folowing \cite{Pie} we define:

\begin{definition}
	A positive number $\tau$ is a singular value of operador $T:H\rightarrow K$ if there exist unitary vectors $x \in H$ and $y\in K$ such that 
	\[
	T(x)=\tau y  \, \, \,\mbox{and} \, \, \,  T^*(y)=\tau	x \,,
	\]
	where $T^*:K\rightarrow H$ is the adjoint operador of $T$.
\end{definition}

We have:
\[
\begin{array}{rcl}
T^*(T(x))&=& T^*(\tau y)=\tau (T^*(y))=\tau^2 x \nonumber \\
T(T^*(y))&=& T(\tau x)=\tau (T(x))=\tau^2 y\, .\\
\end{array}
\]

Thus, $\tau ^2$ is eigenvalue of $(T^*T)$ and $(T T^*)$. We observe that $T$ and $T^*$ have the same singular values. We say that $(x_i)$ is a extended orthonormal sequence if $\langle x_i,x_j\rangle =0$ if $i\neq j$ and each $x_i$ is null or unitary, for all $i,j \in \Nset$. Then, if $H$ and $K$ are separable Hibert spaces and $T : H \rightarrow K$ is a bilinear operator, $T$ has a Schmidt representation if there exists a sequence $(\tau_n) \in c_0$ and extended orthonormal sequences $(x_n)$ in $H$ and $(y_n)$ in $K$, such that
\[
T(x)=\sum^\infty_{n=1}\tau_n<x,x_n> y_n \,,
\]
for all $x, y \in H$.

A Schmidt representation is monotone if $\tau_1 \geq \tau_2 \geq \ldots \geq 0$. We have the following result from Pietsch \cite{Pie}. 

\begin{theorem} Every compact operator $T : H \rightarrow K$ has a monotone Schmidt representation.
\end{theorem}

\section{Singular Value of a Bilinear Operator}

In this section we define singular value of a bilinear operator and we show that every compact nonzero bilinear operator has a singular value. 

\begin{definition}\label{fih1} Let $H_1$, $H_2$ and $K$ be real separable Hilbert spaces. Given $T \in Bil(H_1 \times H_2,K)$, we define the operators  
	\[
	\begin{array}{rcl}
	\phi_{H_1} &:& H_2 \rightarrow \mathfrak{L}(H_1,K) \nonumber \\
	\phi_{H_2} &:& H_1 \rightarrow \mathfrak{L}(H_2,K) \nonumber\\
	\end{array}
	\]
	where, given $x \in H_1$, $y \in H_2$, we have
	\[
	\begin{array}{rcl}
	\phi_{H_1}(y)(x) &=& T(x,y) \nonumber\\
	\phi_{H_2}(x)(y) &=& T(x,y)\nonumber.\\
	\end{array}\]
\end{definition}

\begin{definition} A positive number $\tau$ is a singular value of $T \in Bil(H_1 \times H_2,K)$ if, there exist normalized vectors $x_0 \in H_1$, $y_0\in H_2$ and $z_0 \in K$, such that
	\[
	\begin{array}{rcl}
	T(x_0,y_0) & = &\tau z_0 \\
	\phi^*_{H_1}(y_0)(z_0) & =& \tau x_0 \\
	\phi^*_{H_2}(x_0)(z_0) & = &\tau y_0 
	\end{array}
	\]
where $\phi_{H_1}(y_0) :H_1 \rightarrow K$ and $\phi_{H_2}(x_0) :H_2 \rightarrow
	K$ are given in Definition \ref{fih1}, and $\phi^*_{H_1}(y_0) : K \rightarrow H_1$, $\phi^*_{H_2}(x_0) :K \rightarrow H_2$ are the adjoint operators of $\phi_{H_1}(y_0)$ and $\phi_{H_2}(x_0)$, respectively.
	
	In this case we say that $(x_0 \, , y_0 \, , z_0)$ is a triple of singular vectors associated to the singular value $\tau$.
\end{definition}

\begin{remark}
	In Examples \ref{exschmidt01} and \ref{exschmidt02} next we will see that we can have more than one triple $(x_0 \,, y_0 \,, z_0)$ of singular vectors associated with the same singular value $\tau$.
\end{remark}

\begin{remark} We can see that 
	\begin{align*}
	\langle T(x_0,y_0),z_0\rangle_K = \langle \phi_{H_1}(y_0)(x_0),z_0\rangle_K =
	\langle x_0,\phi^*_{H_1}(y_0)(z_0)\rangle_{H_1} = \langle x_0,\tau x_0\rangle_{H_1} = \tau \\
	\langle T(x_0,y_0),z_0\rangle_K = \langle \phi_{H_2}(x_0)(y_0),z_0\rangle_K =
	\langle y_0,\phi^*_{H_2}(x_0)(z_0)\rangle_{H_2} = \langle y_0,\tau y_0\rangle_{H_2} = \tau\,
	\vspace{-0.5cm}.
	\end{align*}
Thus, we have that $\tau^2$  is eigenvalue of:
	\begin{align}
	\phi_{H_1}(y_0) \circ \phi^*_{H_1}(y_0)&: K \rightarrow K & \phi_{H_2}(x_0) \circ
	\phi^*_{H_2}(x_0)&: K \rightarrow K \nonumber\\
	\phi^*_{H_1}(y_0) \circ \phi_{H_1}(y_0)&: H_1 \rightarrow H_1 & \phi^*_{H_2}(x_0) \circ
	\phi_{H_2}(x_0)&: H_2 \rightarrow H_2 \vspace{-1cm}\label{sing02}\,.
	\end{align}
Indeed, just note that
	\begin{equation}\phi_{H_1}(y_0) \circ
	\phi^*_{H_1}(y_0)(z_0)=\phi_{H_1}(y_0)(\tau x_0)=T(\tau
	x_0,y_0)=\tau T(x_0,y_0)=\tau^2 z_0 \, \vspace{-0.3cm}.\nonumber
	\end{equation}
Similar reasoning applies to the other linear operators in \eqref{sing02}.
\end{remark}

Given a bilinear operator $T \in Bil(H_1 \times H_2,K)$,  will it have a singular value? As in the linear case, we will show in Theorem \ref{sing03} that if the bilinear operator is compact, then we guarantee the existence of at least one singular value.

\begin{lemma}\label{fih1compacto}
	Let $T\in Bil(H_1 \times H_2, K)$. If $T$ is compact, then the operators $\displaystyle\phi_{H_1}$ and $\phi_{H_2}$ given in the Definition \ref{fih1} are linear and compact operators.
\end{lemma}
\begin{proof} Let $y,y'\in H_2$ and $\alpha \in \Rset$. We have:
	\[
	\begin{array}{rcl}
	\phi_{H_1}(\alpha y + y')(x)&=&T(x,\alpha y + y')\\
	&=& \alpha T(x, y) + T(x,y')\\
	&=& (\alpha \phi_{H_1}(y) +\phi_{H_1}( y'))(x),\\
	\end{array}
	\]
	for all $x\in H_1$, then, $\phi_{H_1}(\alpha y + y')=\alpha
	\phi_{H_1}(y) +\phi_{H_1}( y')$. Thus, $\phi_{H_1}$ is linear. In a similar way we show that $\phi_{H_2}$ is also linear.
	
	If $T \in Bil(H_1 \times H_2, K)$ is compact, $T(x, \cdot):
	H_2 \rightarrow K$ and $T( \cdot,y): H_1 \rightarrow K$ are compact linear operators for all $x \in H_1$, $y\in H_2$ fixed. Indeed, suppose $\phi_{H_1}$ is not compact. So, there is a bounded sequence $(y_n) \subset H_2$, such that $(\phi_{H_1}(y_n))_{n \in \Nset}$ has no convergent subsequence, that is, if $(\phi_{ H_1}(y_{n_k}))_{k \in \Nset}$ is a subsequence of $ (\phi_{H_1}(y_n))_ {n \in \Nset}$, then there is $\epsilon > 0$ such that
	
	\begin{equation}\label{epsilon}
	\epsilon \leq \| \phi_{H_1}(y_{n_{i}}) -
	\phi_{H_1}(y_{n_{j}})\|_{_{\mathfrak{L}(H_1,K)}},
	\end{equation}
	for all $i\neq j \in \Nset$.
	
	Note that for each fixed $k \in \Nset$, $ \phi_{H_1}(y_{n_ {k}}): H_1 \rightarrow K$ is compact because
	\[
	\phi_{H_1}(y_{n_{k}})=T(\cdot, y_{n_{k}})\,.
	\]
	
	Let $x \in H_1$, tal que $\|x\|_{_{H_1}} = 1$. We have that $(y_{n_k})_{_{k\in \Nset}}$ is a bounded sequence, and since $T(x,\cdot): H_2 \rightarrow K$ is compact, then 	$(T(x,y_{n_k}))_{_{k \in \Nset}}$ has a convergent subsequence in $K$, that is, there exists a subsequence $(y_{n_{k_{s}}})_{_{s\in \Nset}} \subset (y_{n_k})_{_{k \in \Nset}}$ such that 
	\[T(x,y_{n_{k_{s}}}) \rightarrow z \,\, \mbox{com } \, \, s
	\rightarrow \infty,
	\]
	where $z \in K$ depends on $x \in H_1$.
	
	Note that $(\phi_{H_1}(y_{n_{k_{s}}}))_{s \in \Nset} \subset
	(\phi_{H_1}(y_{n_k}))_{k\in \Nset}$ and $\phi_{H_1}(y_{n_{k_{s}}})(x)=T(x,y_{n_{k_{s}}}) \rightarrow z$,
	when $ s \rightarrow \infty$. Then, we have
	\[
	\|\phi_{H_1}(y_{n_{k_{i}}})(x)-
	\phi_{H_1}(y_{n_{k_{j}}})(x)\|_{_{K}} \rightarrow 0\,,
	\]
	when $i,j \rightarrow \infty$. Thus,
	\[
	\|\phi_{H_1}(y_{n_{k_{i}}})-
	\phi_{H_1}(y_{n_{k_{j}}})\|_{_{\mathfrak{L}(H_1,K)}}=\sup_{\|x\|_{_{H_1}}=1}
	\|\phi_{H_1}(y_{n_{k_{i}}})(x)-
	\phi_{H_1}(y_{n_{k_{j}}})(x)\|_{_{K}} \rightarrow 0\,,
	\]
	when $i,j \rightarrow \infty$. From \eqref{epsilon}, we get a contradiction. Therefore, $\phi_{H_1}$ is compact. Similarly, $\phi_{H_2}$ is shown to be compact.
\end{proof}

With the Lemma \ref{fih1compacto} we can prove the following theorem.

\begin{theorem}\label{sing03}
	Let $T \in Bil(H_1 \times H_2,K)$ be a compact operator. If $T \neq 0$, then $\tau = \|T\|$ is a singular valor of $T$.
\end{theorem}

\begin{proof} We know that $\|T\| =\displaystyle \sup_{\|x\|_{_{H_1}}=1\, , \,\|y\|_{_{H_2}}=1} \|T(x,y)\|_{_{K}}$. Thus, there exists a sequence $\{(x_n,y_n)\}_{_{n\in\Nset}} \in H_1 \times H_2$, with com $\|x\|_{_{H_1}}=\|y\|_{_{H_2}}=1$ and $\displaystyle\lim_{n\rightarrow \infty} \|T(x_n,y_n)\|_{_{K}} =
	\|T\| = \tau$.
	
	Let $(z_n)_{_{n\in \Nset}} \in K$ be the sequence defined by $z_n =
	\tau^{-1} T(x_n,y_n)$. For all $n\in \Nset$, we have $\|z_n\|_{_{K}} \leq 1$, since
	\[
	\begin{array}{rcl}
	\|z_n\|_{_{K}}&=&\|\tau^{-1}T(x_n,y_n)\|_{_{K}}=\tau^{-1}\|T(x_n,y_n)\|_{_{K}}\\
	&\leq& \tau^{-1}\|T\| \|x_n\|_{_{H_1}} \|y_n\|_{_{H_2}}\\
	&=&\tau^{-1} \tau =1 \,.\\
	\end{array}
	\]
	
	Since $T$ is compact and $\{(x_n,y_n)\}_{_{{n\in\Nset}}}$ is a bounded sequence in $  H_1 \times H_2$, the sequence $T(x_n,y_n)$ has a convergent subsequence in $K$, that is, there exists $z \in K$ such that $z_{n_{k}} = \lambda^{-1} T(x_{n_{k}},y_{n_{k}}) \rightarrow z $, when $k \rightarrow \infty$. Note that $\|z\|_{_{K}}=1$, since
	\[
	\begin{array}{rcl}
	\|z\|_{_{K}} & =& \displaystyle\lim_{k\rightarrow \infty}\|z_{n_{k}}\|_{_{K}} = \displaystyle\lim_{k\rightarrow \infty} \|\tau^{-1} T(x_{n_{k}},y_{n_{k}})\|_{_{K}} \vspace{0.1cm}\\
	& =& \tau^{-1} \displaystyle\lim_{k\rightarrow \infty} \|T(x_{n_{k}},y_{n_{k}})\|_{_{K}} \vspace{0.1cm}\\
	&=& \tau^{-1} \tau = 1 \,.\\
	\end{array}
	\]
	
	Now, as $\|y_{n_{k}}\|_{_{H_2}} = 1$, for all $k\in \Nset$, then $(y_{n_{k}})_{_{k \in \Nset}}$ is a bounded sequence in $H_2$. By Lemma \ref{fih1compacto} the operator 	$\phi_{H_1} : H_2 \rightarrow \mathfrak{L}(H_1,K)$ is compact, then 	$(\phi_{H_1}(y_{n_{k}}))_{_{k\in \Nset}}$ has a convergent subsequence. Thus, there exists a subsequence $(y_{n_{k_{s}}})_{_{s\in\Nset}} \subset (y_{n_{k}})_{_{k\in \Nset}}$ such that $(\phi_{H_1}(y_{n_{k_{s}}}))_{_{s \in \Nset}}$ converges in 	$\mathfrak{L}(H_1,K)$, that is, there exists $\phi_1 \in \mathfrak{L}(H_1,K)$ such that $\phi_{H_1}(y_{n_{k_{s}}})	\rightarrow \phi_1 $, when $s \rightarrow \infty$.
	
	We note that, given $y \in H_2$ such that $\|y\|_{_{H_2}}=1$, we have $\phi_{H_1}(y) \leq T$. Indeed, since $\phi_{H_1}(y) \in \mathfrak{L}(H_1,K)$, 
	\[
	\begin{array}{rcl}
	\|\phi_{H_1}(y)\|_{_{\mathfrak{L}(H_1,K)}}&=&\displaystyle
	\sup_{\|x\|_{_{H_1}}=1} \|
	\phi_{H_1}(y)(x)\|_{_{K}}=\displaystyle\sup_{\|x\|_{_{H_1}}=1} \| T(x,y)\|_{_{K}}\\
	&\leq&\displaystyle\sup_{\|x\|_{_{H_1}}=1 , \|y\|_{_{H_2}}=1} \| T(x,y)\|_{_{K}}=\|T|.\\
	\end{array}
	\]
	
	As $ \phi_{H_1}(y_{n_{k_{s}}}) \rightarrow \phi_1 $ and $\|\phi_{H_1}(y_{n_{k_{s}}})\|_{_{\mathfrak{L}(H_1,K)}} \leq T $, for all $s \in \Nset$, then $\|\phi_1 \|_{_{\mathfrak{L}(H_1,K)}} \leq T$. Thus,  
	\[
	\begin{array}{rcl}
	\|\phi^*_{H_1}(y_{n_{k_{s}}})(z_{n_{k_{s}}}) - \tau x_{n_{k_{s}}}\|_{_{H_1}}^2 & =& \langle \phi^*_{H_1}(y_{n_{k_{s}}})(z_{n_{k_{s}}}) - \tau x_{n_{k_{s}}},\phi^*_{H_1}(y_{n_{k_{s}}})(z_{n_{k_{s}}}) - \tau x_{n_{k_{s}}} \rangle \vspace{0.1cm}\\
	&=& \langle \phi^*_{H_1}(y_{n_{k_{s}}})(z_{n_{k_{s}}}),\phi^*_{H_1}(y_{n_{k_{s}}})(z_{n_{k_{s}}})\rangle \\
	& &-2\tau \langle \phi^*_{H_1}(y_{n_{k_{s}}})(z_{n_{k_{s}}}),x_{n_{k_{s}}}\rangle + \tau^2\vspace{0.1cm}\\
	&=&\|\phi^*_{H_1}(y_{n_{k_{s}}})(z_{n_{k_{s}}})\|_{_{H_1}}^2 - 2 \tau \langle z_{n_{k_{s}}},\phi_{H_1}(y_{n_{k_{s}}})(x_{n_{k_{s}}})\rangle +\tau^2 \\
	&=&\|\phi^*_{H_1}(y_{n_{k_{s}}})(z_{n_{k_{s}}})\|_{_{H_1}}^2 \\
	& &- 2 \tau \langle \tau^{-1}T(x_{n_{k_{s}}},y_{n_{k_{s}}}),T(x_{n_{k_{s}}},y_{n_{k_{s}}}))\rangle + \tau^2\\
	&\leq& \|\phi^*_{H_1}(y_{n_{k_{s}}})\|_{_{\mathfrak{L}(K,H_1)}}^2\|z_{n_{k_{s}}}\|_{_{K}}^2-2\|T(x_{n_{k_{s}}},y_{n_{k_{s}}})\|^2_{_{K}} + \tau^2\\
	&\leq&\|\phi^*_{H_1}(y_{n_{k_{s}}})\|_{_{\mathfrak{L}(K,H_1)}}^2-2\|T(x_{n_{k_{s}}},y_{n_{k_{s}}})\|_{_{K}}^2 + \tau^2\\
	&=&\|\phi_{H_1}(y_{n_{k_{s}}})\|_{_{\mathfrak{L}(K,H_1)}}^2-2\|T(x_{n_{k_{s}}},y_{n_{k_{s}}})\|_{_{K}}^2 + \tau^2\\
	&\leq& \|T\|^2 +\tau^2  -2 \|T(x_{n_{k_{s}}},y_{n_{k_{s}}})\|_{_{K}}^2\\
	&=& 2\tau^2  -2 \|T(x_{n_{k_{s}}},y_{n_{k_{s}}})\|_{_{K}}^2\, .\\
	\end{array}
	\]
	Then,
	\[
	\|\phi^*_{H_1}(y_{n_{k_{s}}})(z_{n_{k_{s}}}) - \tau
	x_{n_{k_{s}}}\|_{_{H_1}}^2\leq 2\tau^2  -2
	\|T(x_{n_{k_{s}}},y_{n_{k_{s}}})\|_{_{K}}^2 \rightarrow 0 \, ,
	\]
	when $s \rightarrow \infty$. Since $\phi_{H_1}(y_{n_{k_{s}}}) \rightarrow \phi_1$, then  $\|	\phi_1- \phi_{H_1}(y_{n_{k_{s}}})\|_{_{\mathfrak{L}(H_1,K)}} \rightarrow 0$, when $s \rightarrow \infty$, giving
	\[
	\| \phi^*_1 -
	\phi^*_{H_1}(y_{n_{k_{s}}})\|_{_{\mathfrak{L}(K,H_1)}}=\|(\phi_1-\phi_{H_1}(y_{n_{k_{s}}}))^*\|_{_{\mathfrak{L}(K,H_1)}}=\|\phi_1-\phi_{H_1}(y_{n_{k_{s}}})\|_{_{\mathfrak{L}(H_1,K)}} \rightarrow 0 \,.
	\]
	
	Thus, we have 
	\[
	\begin{array}{rcl}
	\|\phi^*_{1}(z) - \tau x_{n_{k_{s}}}\|_{_{H_1}} &=& \|\phi^*_{1}(z) - \phi^*_{H_1}(y_{n_{k_{s}}})(z_{n_{k_{s}}}) + \phi^*_{H_1}(y_{n_{k_{s}}})(z_{n_{k_{s}}}) - \tau x_{n_{k_{s}}}\|_{_{H_1}} \vspace{0.1cm}\\
	&\leq&\|\phi^*_{1}(z) - \phi^*_{H_1}(y_{n_{k_{s}}})(z_{n_{k_{s}}})\|_{_{H_1}} + \|\phi^*_{H_1}(y_{n_{k_{s}}})(z_{n_{k_{s}}}) - \tau x_{n_{k_{s}}}\|_{_{H_1}} \vspace{0.1cm}\\
	& = &\|\phi^*_{1}(z) - \phi^*_{1}(z_{n_{k_{s}}}) + \phi^*_{1}(z_{n_{k_{s}}}) - \phi^*_{H_1}(y_{n_{k_{s}}})(z_{n_{k_{s}}})\|_{_{H_1}}\\
	& & + \|\phi^*_{H_1}(y_{n_{k_{s}}})(z_{n_{k_{s}}}) - \tau x_{n_{k_{s}}}\|_{_{H_1}} \vspace{0.1cm}\\
	& \leq &\|\phi^*_{1}(z) - \phi^*_{1}(z_{n_{k_{s}}})\|_{_{H_1}} + \|\phi^*_{1}(z_{n_{k_{s}}}) - \phi^*_{H_1}(y_{n_{k_{s}}})(z_{n_{k_{s}}})\|_{_{H_1}} \\
	& &+ \|\phi^*_{H_1}(y_{n_{k_{s}}})(z_{n_{k_{s}}}) - \tau x_{n_{k_{s}}}\|_{_{H_1}} \vspace{0.1cm}\\
	&\leq&\|\phi^*_{1}\|_{_{\mathfrak{L}(K,H_1)}} \|z - z_{n_{k_{s}}}\|_{_{K}} + \|\phi^*_{1} - \phi^*_{H_1}(y_{n_{k_{s}}})\|_{_{\mathfrak{L}(K,H_1)}} \|z_{n_{k_{s}}}\|_{_{K}} \\
	& &+ \|\phi^*_{H_1}(y_{n_{k_{s}}})(z_{n_{k_{s}}}) - \tau x_{n_{k_{s}}}\|_{_{H_1}} \,. \\
	\end{array}
	\]
	
	As $(z_{n_{k_{s}}})_{s\in \Nset} \subset (z_{n_{k}})_{k\in \Nset}$ and $z_{n_{k}}\rightarrow z$, then $z_{n_{k_{s}}} \rightarrow z$. Since the terms $\|z - z_{n_{k_{s}}}\|_{_{K}}$, $\|\phi^*_{1} -\phi^*_{H_1}(y_{n_{k_{s}}})\|_{_{\mathfrak{L}(K,H_1)}}$ and $\|\phi^*_{H_1}(y_{n_{k_{s}}})(z_{n_{k_{s}}}) - \tau x_{n_{k_{s}}}\|_{_{H_1}}$ converge to zero, we get $\tau x_{n_{k_{s}}} \rightarrow \phi^*_{1}(z)$, or $x_{n_{k_{s}}}
	\rightarrow \tau^{-1} \phi^*_{1}(z)$. Define $x = \tau^{-1}	\phi^*_{1}(z) \in H_1$, thus, $x_{n_{k_{s}}} \rightarrow x$.
	
	The same calculations are valid for the operator $\phi_{H_2}: H_1 \rightarrow \mathfrak {L}(H_2, K)$, where an operator $\phi_2: H_2 \rightarrow K$ is obtained such that 
	$\|\phi_2^*(z) - \tau y_{n_{k_{s}}}\|_{_{H_2}} \rightarrow 0$. Define $y = \tau^{-1} 	\phi_2^*(z) \in H_2$, then $y_{n_{k_{s}}} \rightarrow y$. For all $h_1 \in H_1$ we have
	\[
	\phi_1(h_1) = \lim_{s \rightarrow \infty}
	\phi_{H_1}(y_{n_{k_{s}}})(h_1) = \lim_{s \rightarrow \infty}
	T(h_1,y_{n_{k_{s}}}) = T(h_1,y) = \phi_{H_1}(y)(h_1)\,,
	\]
	implying $\phi_1 = \phi_{H_1}(y)$. With the same reasoning we get $\phi_2 =\phi_{H_2}(x)$. Since 
	\[
	x = \tau^{-1}\phi_1^* (z) = \tau^{-1} \phi_{H_1}^*(y)(z) \,,
	\]
	then $\phi_{H_1}^*(y)(z) = \tau x$, and in the same way, $\phi_{H_2}^*(x)(z) = \tau y$. Since $\tau z_{n_{k_{s}}} = T(x_{n_{k_{s}}},y_{n_{k_{s}}})$, taking the limit, we get $T(x,y)=\tau z$. It is easy to show that $\|x\|_{_{H_1}}=1$ and $\|y\|_{_{H_2}}=1$, giving that $\|T\|$ is a singular value of $T$.
\end{proof}

\section{Schmidt representation of a bilinear operator}

If $T \in Bil(H_1 \times H_2,K)$ is compact, then $\|T\|$ is a singular value. Now, we can define the Schmidt representation of an operador $T \in Bil(H_1 \times H_2,K)$.
\begin{definition}\label{defschmidt} The summation 
	\[
	\sum_{i=1}^{\infty} \tau_i \langle \cdot,x_i\rangle \langle \cdot,y_i\rangle z_i
	\]
	is a Schmidt representation of the operador $T \in Bil(H_1 \times H_2,K)$ if the following conditions are satisfied:
	\begin{enumerate}
		\item $(\tau_i) \in c_0$;
		\item $(x_i)$, $(y_i)$ and $(z_i)$ are orthonormal extended sequences of $H_1$, $H_2$ and $K$, respectively;
		\item $T(x,y)=\displaystyle\sum_{i=1}^{\infty} \tau_i \langle x,x_i\rangle \langle y,y_i\rangle z_i$ for all 
		$(x,y) \in H_1 \times H_2$;
		\item If $x_i \neq 0$, $y_i \neq 0$ and $z_i \neq 0$, then it follows that $\langle T(x_i,y_i),z_i\rangle = \tau_i$. Otherwise, if at least one of the elements $x_i$, $y_i$ or $z_i$ is zero, the value of the corresponding coefficient $\tau_i$ has no effect. Therefore we naturally assume that $\tau_i = 0$. Then, we get $\langle T(x_i,y_i),z_i\rangle = \tau_i$ for all $i\in \Nset$ \, .
	\end{enumerate}
\end{definition}
The Schmidt representation is monotone if $\tau_1 \geq \tau_2 \geq \ldots \geq 0$.

If $H$ and $K$ are Hilbert spaces of finite dimension, and $ T: H \rightarrow K$ is a linear operator with rank $r$, so the operator $T$ has $r$ singular values. We will see in the Example \ref{exschmidt01} that this does not occur for bilinear operators. We saw that every compact linear operator $T \in \mathfrak{L}(H, K)$ has a monotonous Schmidt representation. We will see in Examples \ref{exschmidt01} and \ref{exschmidt02} that not all compact bilinear operator $T \in Bil(H_1 \times H_2, K)$ has a Schmidt representation.

\begin{example}\label{exschmidt01}
	Let $T: \Rset^3 \times \Rset^2 \rightarrow \Rset^4$ defined in the following way: if $x=(a_1,a_2,a_3) \in \Rset^3$, $y=(b_1,b_2) \in \Rset^2$ and $z=(u_1,u_2,u_3,u_4)\in \Rset^4$, then:
	\[
	T(x,y)=(2a_1b_1, \, 3a_2b_2,\, 0,\, 0)\,.
	\]
	It is simple to verify that $T\in Bil(\Rset^3 \times \Rset^2,\Rset^4)$. Let $y\in \Rset^2$ fixed. We have $\phi_{\Rset^3}(y):\Rset^3 \rightarrow \Rset^4$ is given by $\phi_{\Rset^3}(y)(x)=T(x,y)$. Let $\{e_1,e_2,e_3\}$ be the canonical basis of $\Rset^3$, then we have 
	\[
	\left \{ \begin{array}{rcl}
	\phi_{\Rset^3}(y)(e_1)&=&(\,2 b_1,\, 0,\,0, \, 0)\\
	\phi_{\Rset^3}(y)(e_2)&=&(\,0,\,3 b_2,\,0,\,0)\\
	\phi_{\Rset^3}(y)(e_3)&=&(\,0,\,0,\,0,\,0).\\
	\end{array}\right.
	\]
	That is,
	\[
	A=\left [\begin{array}{ccc}
	2b_1 & 0 & 0\\
	0 & 3 b_2 & 0 \\
	0 & 0 & 0 \\
	0 & 0 & 0 \\
	\end{array} \right ]
	\]
	is the matrix of $\phi_{\Rset^3}(y):\Rset^3 \rightarrow \Rset^4$ whence the canonical basis of $\Rset^3$ e $\Rset^4$ is fixed. To find $\phi^*_{\Rset^3}(y): \Rset^4 \rightarrow \Rset^3$, we just calculate the transposed matrix of $A$. Then, 
	\[
	\phi^*_{\Rset^3}(y)(z)=(\,2b_1u_1,\, 3b_2u_2, \,0)\,.
	\]
	Let $x\in \Rset^3$ fixed, in an analogous way we show that
	\[
	\phi^*_{\Rset^2}(x)(z)=(\,2a_1u_1,\, 3a_2u_2) \,.
	\]
	To find the singular values of $T$ we have to solve the system:
	\begin{equation}
	\left \{ \begin{array}{rcl}
	T(x_0,y_0)&=&\tau z_0\\
	\phi^*_{\Rset^3}(y_0)(z_0)&=&\tau x_0\\
	\phi^*_{\Rset^2}(x_0)(z_0)&=&\tau y_0 \, ,\\
	\end{array}\right.
	\label{sys1}
	\end{equation}
	where $(x_0, \, y_0, \, z_0)$ is a triple of singular vectors associated to the singular value $\tau$. Replacing the operators in (\ref{sys1}), we get
	\[
	\left \{ \begin{array}{rclcrcl}
	2 a_1 b_1 & = & \tau u_1 & & 3 a_2 b_2 & = & \tau u_2\\
	0& =& \tau u_3 & & 0& = &\tau u_4 \\
	2 b_1 u_1 & =& \tau a_1 & & 3 b_2 u_2&= &\tau a_2\\
	0& =& \tau a_3 & & 2 a_1 u_1 &= &\tau b_1\\
	3 a_2 u_2& = &\tau b_2 & & (a_1)^2+(a_2)^2+(a_3)^2&= &1\\
	(b_1)^2+(b_2)^2&= &1 & & (u_1)^2+(u_2)^2+(u_3)^2+(u_4)^2 &= &1\\
	\tau & > & 0\, . & & & &\\
	\end{array}\right.
	\]
	
	Using a computational algebraic system as Maple or MatLab, we solve the system above and obtain the following singular values, with their respective associated triple of singular vectors.
	
	$\tau=2$:
	\begin{align*}
	\left (x=(1,0,0) , y=(1,0), z=(1,0,0,0)\right ) \\
	\left (x=(1,0,0),y=(-1,0) , z=(-1,0,0,0)\right ) \\
	\left ( x=(-1,0,0), y=(1,0), z=(-1,0,0,0)\right ) \\
	\left (x=(-1,0,0),y=(-1,0),z=(1,0,0,0)\right) \,.
	\end{align*}
	$\tau=3$:
	\begin{align*} 
	(x=(0,1,0) , y=( 0, 1) , z=( 0,  1,	0, 0)) \\
	(x=(0,1,0), y=(0,-1) ,z=(0, -1, 0, 0))\\
	(x=(0, -1, 0, ), y=(0, 1) , z=( 0,  -1, 0, 0)) \\
	(x=( 0, -1, 0),y=( 0, -1) , z=( 0, 1,  0, 0)) \,.
	\end{align*}
	
	$\tau = \displaystyle \frac{6}{13} \sqrt{13}$:
	\begin{eqnarray*}
		&&x=(\displaystyle \frac{3}{13} \sqrt{13}, \displaystyle
		\frac{2}{13} \sqrt{13}), 0) , y=(\displaystyle \frac{3}{13}
		\sqrt{13}, \displaystyle \frac{2}{13} \sqrt{13}), z=(\displaystyle
		\frac{3}{13}
		\sqrt{13}, \displaystyle \frac{2}{13} \sqrt{13}, 0, 0)\\
		&&x=(\displaystyle \frac{3}{13} \sqrt{13}, \displaystyle
		-\frac{2}{13} \sqrt{13}, 0) , y=( \displaystyle \frac{3}{13}
		\sqrt{13},\displaystyle -\frac{2}{13} \sqrt{13}), z =(\displaystyle
		\frac{3}{13} \sqrt{13}, \displaystyle \frac{2}{13} \sqrt{13},0,0)\\
		&&x=(-\displaystyle \frac{3}{13} \sqrt{13}, \displaystyle
		\frac{2}{13} \sqrt{13},  0) , y=( -\displaystyle \frac{3}{13}
		\sqrt{13},\displaystyle \frac{2}{13} \sqrt{13}) ,z =(\displaystyle
		\frac{3}{13} \sqrt{13}, \displaystyle \frac{2}{13} \sqrt{13},  0,
		0)\\
		&&x=( -\displaystyle \frac{3}{13} \sqrt{13}, -\displaystyle
		\frac{2}{13} \sqrt{13},  0), y = (-\displaystyle \frac{3}{13}
		\sqrt{13}, -\displaystyle \frac{2}{13} \sqrt{13}),z=( \displaystyle
		\frac{3}{13} \sqrt{13}, \displaystyle \frac{2}{13} \sqrt{13},  0,
		0)\\
		&&x=( -\displaystyle \frac{3}{13} \sqrt{13}, \displaystyle
		\frac{2}{13} \sqrt{13},0) , y = (\displaystyle \frac{3}{13}
		\sqrt{13},\displaystyle \frac{2}{13} \sqrt{13}),z=(-\displaystyle
		\frac{3}{13} \sqrt{13}, \displaystyle \frac{2}{13} \sqrt{13},  0,
		0)\\
		&&x=( -\displaystyle \frac{3}{13} \sqrt{13},  -\displaystyle
		\frac{2}{13} \sqrt{13},  0) , y = (\displaystyle \frac{3}{13}
		\sqrt{13},  -\displaystyle \frac{2}{13} \sqrt{13}),
		z=(-\displaystyle \frac{3}{13} \sqrt{13}, \displaystyle \frac{2}{13}
		\sqrt{13}, 0,0)\\
		&&x=(\displaystyle \frac{3}{13} \sqrt{13}, \displaystyle
		\frac{2}{13} \sqrt{13}, = 0), y=( -\displaystyle \frac{3}{13}
		\sqrt{13}, \displaystyle \frac{2}{13} \sqrt{13}), z=( -\displaystyle
		\frac{3}{13} \sqrt{13}, \displaystyle \frac{2}{13} \sqrt{13}, 0,
		0)\\
		&&x=(\displaystyle \frac{3}{13} \sqrt{13}, -\displaystyle
		\frac{2}{13} \sqrt{13},  0) , y=(-\displaystyle \frac{3}{13}
		\sqrt{13}, -\displaystyle \frac{2}{13} \sqrt{13}),z=( -\displaystyle
		\frac{3}{13} \sqrt{13},\displaystyle \frac{2}{13} \sqrt{13},  0, 0) \,.
	\end{eqnarray*}
	
	Note that if
	\[
	\tau_1=3 \, \, \mbox{ and } \, \,x_1=(\,0,\,1,\,0) , y_1=(\, 0,\, 1) , z_1=(\, 0,\, 1,\, 0,\, 0)\]
	\[
	\tau_2= 2\, \, \mbox{ and } \, \, x_2=(\,1,\,0,\,0) , y_2=(\,1,\,0), z_2=(\,1,\,0,\,0,\,0 ) \,,
	\]
	then,	
	\[
	T(x,y)=\sum^2_{i=1}\tau_i <x,x_i><y,y_i>z_i
	\]
	for all $(x,y) \in \Rset^3 \times \Rset^2$. We have that the rank of $T$ is $2$ and $T$ has three singular values.
\end{example}

\begin{example}\label{exschmidt02}
	Let $T: \Rset^3 \times \Rset^2 \rightarrow \Rset^4$ defined in the following way: if $x=(a_1,a_2,a_3) \in \Rset^3$, $y=(b_1,b_2) \in \Rset^2$ and $z=(u_1,u_2,u_3,u_4)\in \Rset^4$, then:
	\[
	T(x,y)=(a_1b_1, \, b_1(a_1+a_2), \, b_1 a_1, \, b_2(a_1+a_3))\,.
	\]
	Analogously to Example \ref{exschmidt01}, we show that
	\[
	\phi^*_{\Rset^3}(y)(z)=(b_1u_1+b_1u_2+b_1u_3+b_2u_4,\, b_1u_2,\,
	b_2u_4)\,,
	\]
	\[
	\phi^*_{\Rset^2}(x)(z)=(a_1u_1+(a_1+a_2)u_2+a_1u_3,\,
	(a_1+a_3)u_4)\,.
	\]
	
	After some calculations, similarly to Example \ref{exschmidt01}, we have that the singular values of $T$, with their respective associated triple of singular vectors, are:\\
	
	$\tau=\displaystyle \sqrt{2}$
	\begin{eqnarray*}
		&&x= (\displaystyle\frac{1}{2}\sqrt{2},0,\displaystyle\frac{1}{2}\sqrt{2}), y=(0,1), z=(0,0,0,1)\\
		&&x=(-\displaystyle\frac{1}{2}\sqrt{2},0,-\displaystyle\frac{1}{2}\sqrt{2}), y=(0,1), z=(0,0,0,-1)\\
		&&x=(-\displaystyle\frac{1}{2}\sqrt{2},0,-\displaystyle\frac{1}{2}\sqrt{2}), y=(0,-1), z=(0,0,0,1)\\
		&&x=(\displaystyle\frac{1}{2}\sqrt{2},0,\displaystyle\frac{1}{2}\sqrt{2}),
		y=(0,-1), z=(0,0,0,-1) \,.
	\end{eqnarray*}
	$\tau = \displaystyle \sqrt{2+\sqrt{2}}$
	\begin{eqnarray*}
		&&x=(\displaystyle\frac{1}{2}\sqrt{2+\sqrt{2}},\frac{1}{2}(-1+\sqrt{2})\sqrt{2+\sqrt{2}},0), y=(1,0), z=(\frac{1}{2},\frac{1}{2}\sqrt{2},\frac{1}{2},0)\\
		&&x=(-\displaystyle\frac{1}{2}\sqrt{2+\sqrt{2}},-\frac{1}{2}(-1+\sqrt{2})\sqrt{2+\sqrt{2}},0), y=(1,0), z=(-\frac{1}{2},-\frac{1}{2}\sqrt{2},-\sqrt{1}{2},0)\\
		&&x=(-\displaystyle\frac{1}{2}\sqrt{2+\sqrt{2}},-\frac{1}{2}(-1+\sqrt{2})\sqrt{2+\sqrt{2}},0), y=(-1,0), z=(\frac{1}{2},\frac{1}{2}\sqrt{2},\frac{1}{2},0)\\
		&&x=(\frac{1}{2}\sqrt{2+\sqrt{2}},\frac{1}{2}(-1+\sqrt{2})\sqrt{2+\sqrt{2}},0),
		y=(-1,0), z=(-\frac{1}{2},-\frac{1}{2}\sqrt{2},-\frac{1}{2},0) \,.
	\end{eqnarray*}
	
	$\tau = \displaystyle \sqrt{2-\sqrt{2}}$
	\begin{eqnarray*}
		&&x=(-\frac{1}{2}\sqrt{2-\sqrt{2}},-\frac{1}{2}(-1-\sqrt{2})\sqrt{2-\sqrt{2}},0), y=(1,0), z=(-\frac{1}{2},\frac{1}{2}\sqrt{2},-\frac{1}{2},0)\\
		&&x=(\frac{1}{2}\sqrt{2-\sqrt{2}},\frac{1}{2}(-1-\sqrt{2})\sqrt{2-\sqrt{2}},0), y=(1,0) , z=(\frac{1}{2},-\frac{1}{2}\sqrt{2},\frac{1}{2},0),\\
		&&x=(\frac{1}{2}\sqrt{2-\sqrt{2}},\frac{1}{2}(-1-\sqrt{2})\sqrt{2-\sqrt{2}},0), y=(-1,0) , z=(-\frac{1}{2},\frac{1}{2}\sqrt{2},-\frac{1}{2},0)\\
		&&x=(-\frac{1}{2}\sqrt{2-\sqrt{2}},-\frac{1}{2}(-1-\sqrt{2})\sqrt{2-\sqrt{2}},0),
		y=(-1,0), z=(\frac{1}{2},-\frac{1}{2}\sqrt{2},\frac{1}{2},0).
	\end{eqnarray*}
	$\tau=\displaystyle \frac{1}{2}\sqrt{2}$\vspace{-0.3cm}
	\begin{eqnarray*}
		&&x=(0,\frac{1}{2}\sqrt{2},-\frac{1}{2}\sqrt{2}), y=(-\frac{1}{2}\sqrt{2},\frac{1}{2}\sqrt{2}), z=(0,-\frac{1}{2}\sqrt{2},0,-\frac{1}{2}\sqrt{2})\\
		&&x=(0,\frac{1}{2}\sqrt{2},-\frac{1}{2}\sqrt{2}), y=(\frac{1}{2}\sqrt{2},-\frac{1}{2}\sqrt{2}), z=(0,\frac{1}{2}\sqrt{2},0,\frac{1}{2}\sqrt{2})\\
		&&x=(0,-\frac{1}{2}\sqrt{2},\frac{1}{2}\sqrt{2}), y=(-\frac{1}{2}\sqrt{2},\frac{1}{2}\sqrt{2}), z=(0,\frac{1}{2}\sqrt{2},0,\frac{1}{2}\sqrt{2})\\
		&&x=(0,-\frac{1}{2}\sqrt{2},\frac{1}{2}\sqrt{2}),
		y=(\frac{1}{2}\sqrt{2},-\frac{1}{2}\sqrt{2}),
		z=(0,-\frac{1}{2}\sqrt{2},0,-\frac{1}{2}\sqrt{2})\,.
	\end{eqnarray*}
	
	After analyzing the singular values of $T$, we conclude that $T$ does not have a Schmidt representation according to the Definition \ref{defschmidt}.
\end{example}

Examples \ref{exschmidt01} and \ref{exschmidt02} show us that $ T \in Bil (H_1 \times H_2, K)$ be compact is not a sufficient condition for $T$ to have a Schmidt representation. To work around this problem, we will have to require more properties from the operator $T \in Bil (H_1 \times H_2, K)$.

\begin{definition} Let $\tau_1 $ be a singular value of $T \in Bil (H_1 \times H_2, K)$. Then, $\tau_1$ is a ordered singular value of $T$ if:

	\[\left \{ \begin{array}{rcc}
	T(x,y_1) = \tau_1 \langle x,x_1 \rangle z_1 & , & \mbox{for all } x\in H_1\\
	T(x_1,y)=\tau_1 \langle y,y_1 \rangle z_1 & , & \mbox{for all } y\in H_2\\
	\phi^*_{H_1}(y)(z_1) = \tau_1 \langle y,y_1 \rangle x_1 & , & \mbox{for all } y\in H_2
	\end{array}
	\right .
	\]
	for some triple $(x_1 \,, y_1 \,, z_1) $ of singular vectors associated with the singular value $\tau_1$. In this case, $x_1$, $y_1$ and $ z_1 $ are ordered singular vectors associated with the ordered singular value $\tau_ {_{1}}$.
\end{definition}

\begin{example} In Example \ref{exschmidt01}, the singular values $\tau= 2 $ and $\tau=3$ are ordered, and the singular value $\tau=\frac{6}{13}\sqrt{13}$ is not ordered. In Example \ref{exschmidt02}, all singular values are not ordered.
\end{example}

With this definition of ordered singular value, we show in Theorem \ref{existenciadeschmidt} that, if $ T \in Bil(H_1 \times H_2, K)$ is compact and it has a sequence of ordered singular values, such that their respective ordered singulars vectors form an orthonormal sequence of vectors, so $T$ has a Schmidt representation.

\begin{theorem}
	\label{existenciadeschmidt}
	
	Let $T \in Bil(H_1 \times H_2,K)$ be a nonzero and compact operator. Let us define a sequence of operators $(T_k)$ in
	$Bil(H_1 \times H_2,K)$ in the following way. For all $(x,y)\in H_1 \times H_2$, we set
	\[
	T_1(x,y)=T(x,y)\;\;\;\;\; , \;\;\;\;\; \tau_1=||T|| \,,
	\]
and let $(x_1\, , y_1 \, ,z_1)$ be a triple of singular vectors associated to $\tau_1$, that is, they satisfy
	\[
	\begin{array}{rcl}
	T(x_1,y_1) & = &\tau_1 z_1 \\
	\phi^*_{H_1}(y_1)(z_1) & = &\tau_1 x_1 \\
	\phi^*_{H_2}(x_1)(z_1) & = &\tau_1 y_1 \\
	\end{array}
	\]
	where
	\[
	\begin{array}{rlrr}
	\phi_{H_1}(y)&:H_1 \rightarrow K  &  & \phi_{H_1}(y)(x)=T(x,y)  \vspace{0.0cm} \\
	\phi_{H_2}(x)&: H_2 \rightarrow K &  & \phi_{H_2}(x)(y)=T(x,y)\vspace{0.0cm}\\
	\end{array}
	\]
	for all  $(x,y) \in H_1 \times H_2$.
	
	Suppose that $\tau_1$ is a \textbf{ordered singular value} of $T$ and $x_1$, $ y_1 $ and $z_1$ are ordered singular vectors associated to $\tau_1$, that is,
	\[
	\left \{ \begin{array}{rc}
	T(x,y_1)=\tau_1 \langle x,x_1\rangle z_1 &   \\
	T(x_1,y)=\tau_1 \langle y,y_1\rangle z_1 &   \\
	\phi^*_{H_1}(y)(z_1)=\tau_1 \langle y,y_1\rangle x_1 &
	\end{array}
	\right .
	\]
	for all $(x,y) \in H_1 \times H_2$.
	
	Having defined $T_{k-1}$, $\tau_{k-1}$ and $(x_{k-1}, y_{k-1}, z_{k-1})$, for $k\geq 2$, we define
	\[
	T_{k}(x,y)=T_{k-1}(x,y) -\tau_{k-1} \langle x,x_{k-1}\rangle \langle y,y_{k-1}\rangle z_{k-1}\;\;\;, \;\;\; \tau_{k}=||T_{k}|| \,,
	\]
	and $x_k$, $ y_k $ and $z_k$ are ordered singular vectors associated to $\tau_k$. 	Suppose that, for each $k$, we have
	\[
	\left \{ \begin{array}{rc}
	T(x,y_k)=\tau_k \langle x,x_k\rangle z_k &   \\
	T(x_k,y)=\tau_k \langle y,y_k\rangle z_k &   \\
	\phi^*_{H_1}(y)(z_k)=\tau_k \langle y,y_k\rangle x_k &
	\end{array}
	\right .
	\]
	for all $(x,y) \in H_1 \times H_2$.
	
	If $T_k$ is nonzero for all $k$, then $(\tau_k)_k$ is a decreasing sequence of singular values of $T$ converging to zero, such that 
	\begin{equation}\label{ord02}
	T(x,y)=\displaystyle \sum^\infty_{i=1}\tau_i \langle x,x_i\rangle \langle y,y_i\rangle z_i\,,
	\end{equation}
	for all $(x,y) \in H_1 \times H_2$.
\end{theorem}

\begin{proof}
	From Theorem \ref{sing03}, $\tau_1 = \|T\|$ is a singular value of $T$ and let $(x_1\, , y_1 \, ,z_1)$ be a triple of singular vectors associated to $\tau_1$. Let $T_2 \in Bil(H_1 \times H_2,K)$ given by 
	\begin{equation*}
	T_2(x,y)=T(x,y)-\tau_1 \langle x,x_1 \rangle \langle y,y_1 \rangle z_1\, .
	\end{equation*}
	
	Since $T$ is bounded and compact, $T_2$ is compact. If $T_2$ is nonzero, $\tau_2 = \| T_2\|$ is a singular value of $T_2$  with a triple of $(x_2\, , y_2 \, ,
	z_2)$ of singular vectors associated to $\tau_2$. We have:
	\[
	\left \{ \begin{array}{rc}
	T_2(x_2,y_2)  = \tau_2 z_2 \\
	\phi^*_{2H_1}(y_2)(z_2)  = \tau_2 x_2 \\
	\phi^*_{2H_2}(x_2)(z_2)  = \tau_2 y_2 \, .
	\end{array}
	\right . 
	\]
	
	We claim that $x_1 \perp x_2$, $y_1 \perp y_2$ and $z_1 \perp z_2$. Indeed, 
	\[
	\begin{array}{rcl}
	\tau_2 \langle x_1,x_2 \rangle &=& \langle x_1,\tau_2 x_2\rangle = \langle x_1,\phi^*_{2H_1}(y_2)(z_2)\rangle \\
	&=& \langle \phi_{2H_1}(y_2)(x_1),z_2\rangle \langle T_2(x_1,y_2), z_2\rangle \\
	&=& \langle T(x_1,y_2)-\tau_1\langle y_2,y_1\rangle z_1,z_2\rangle \\
	&=& \langle \tau_1 \langle y_2,y_1\rangle z_1-\tau_1 \langle y_2,y_1\rangle z_1, z_2\rangle \\
	&=& 0 \, .\\
	\end{array}
	\]
	In the same way,
	\[
	\begin{array}{rcl}
	\tau_2 \langle y_1,y_2\rangle &=& \langle y_1,\tau_2 y_2\rangle = \langle y_1,\phi^*_{2H_2}(x_2)(z_2)\rangle \\
	&=& \langle \phi_{2H_2}(x_2)(y_1),z_2\rangle = \langle T_2(x_2,y_1), z_2\rangle \\
	&=& \langle T(x_2,y_1)-\tau_1 \langle x_2,x_1\rangle z_1,z_2\rangle \\
	&=& \langle \tau_1 \langle x_2,x_1\rangle z_1-\tau_1 \langle x_2,x_1\rangle z_1, z_2\rangle \\
	&=& 0\, ,\\
	\end{array}
	\]
	and
	\[
	\begin{array}{rcl}
	\tau_2 \langle z_1,z_2\rangle  & = & \langle z_1,\tau_2 z_2\rangle = \langle z_1,T_2(x_2,y_2)\rangle \\
	&=& \langle z_1,T(x_2,y_2)\rangle = \langle z_1,\phi_{H_1}(y_2)(x_2)\rangle \\
	&=& \langle \phi^*_{H_1}(y_2)(z_1),x_2\rangle = \langle \tau_1 \langle y_2,y_1\rangle x_1,x_2\rangle \\
	&=&0 \, .
	\end{array}
	\]
	We have that
	\begin{align}
	\tau_2=\| T_2\|=\sup\{\|T_2(x,y)\|_{_{K}}: \mbox{ where } \|x\|_{_{H_1}}=\|y\|_{_{H_2}}=1, x\in H_1, y\in H_2\}\nonumber\\
	\tau_1=\| T\|=\sup\{\|T(x,y)\|_{_{K}}:\mbox{ where }
	\|x\|_{_{H_1}}=\|y\|_{_{H_2}}=1, x\in H_1, y\in H_2\} \, .
	\label{sing05}
	\end{align}
	
	We have that $T_2(x,y) \perp z_1$ for all $(x,y) \in H_1
	\times H_2$, and this implies that $\tau_1 \geq \tau_2$. Indeed, we have
	\[
	\begin{array}{rcl}
	\langle T_2(x,y),z_1\rangle &=& \langle T(x,y)-\tau_1 \langle x,x_1\rangle \langle y,y_1\rangle z_1,z_1\rangle \\
	&=& \langle T(x,y),z_1\rangle -\tau_1\langle x,x_1\rangle \langle y,y_1\rangle \\
	&=& \langle \phi_{H_1}(y)(x),z_1\rangle -\tau_1 \langle x,x_1\rangle \langle y,y_1\rangle \\
	&=& \langle x,\phi^*_{H_1}(y)(z_1)\rangle - \tau_1 \langle x,x_1\rangle \langle y,y_1\rangle \\
	&=& \langle x,\tau_1\langle y,y_1\rangle x_1\rangle - \langle x,\tau_1 \langle y,y_1\rangle x_1\rangle \\
	&=&0 \, .\\
	\end{array}
	\]
	
	Thus, $T_2(x,y) \perp \tau_1 \langle x,x_1\rangle \langle y,y_1\rangle z_1$, for all $(x,y) \in H_1 \times H_2$. Then, 
	\begin{multline*}
	\|T_2(x,y)\|_{_{K}}^2+\|\tau_1 \langle x,x_1\rangle \langle y,y_1\rangle z_1\|_{_{K}}^2=\|T(x,y)-\tau_1 \langle x,x_1\rangle \langle y,y_1\rangle z_1\\
	+\tau_1 \langle x,x_1\rangle \langle y,y_1\rangle z_1\|_{_{K}}^2=\|T(x,y)\|_{_{K}}^2\,
	.\vspace{-0.8cm}
	\end{multline*}
	
	Then, 
	\begin{equation}\label{sing04}
	\|T_2(x,y)\|_{_{K}}^2+\|\tau_1 \langle x,x_1\rangle \langle y,y_1\rangle z_1\|_{_{K}}^2=\|T(x,y)\|_{_{K}}^2
	\end{equation}
	for all $(x,y) \in H_1\times H_2$. From \eqref{sing04} and \eqref{sing05} we get  $\tau_1 \geq \tau_2$.
	
	Now, since $x_1\perp x_2$,  $y_1\perp y_2$ e $z_1\perp z_2$, we have that 	$T_2(x_2,y_2)=T(x_2,y_2)$, giving $T(x_2,y_2)=\tau_2 z_2$. Note that 
	\[
	\phi^*_{2H_1}(y)(z)=\phi^*_{H_1}(y)(z)-\tau_1
	\langle y,y_1\rangle \langle z,z_1\rangle x_1
	\]
	\[\phi^*_{2H_2}(x)(z)=\phi^*_{H_2}(x)(z)-\tau_1 \langle x,x_1\rangle \langle z,z_1\rangle y_1\, .
	\]
	So,
	\[
	\phi^*_{2H_1}(y_2)(z_2)=\phi^*_{H_1}(y_2)(z_2) \,\,\,\mbox{ and }\,\,\,\phi^*_{2H_2}(x_2)(z_2)=\phi^*_{H_2}(x_2)(z_2)\,,
	\]
	giving that
	\[
	\phi^*_{H_1}(y_2)(z_2)=\tau_2 x_2 \,\,\,\mbox{ and } \,\,\,\phi^*_{H_2}(x_2)(z_2)=\tau_2 y_2\, .
	\]
	Thus, we obtain that $\tau_2$ is a singular value of $T$ with a triple $(x_2\, , y_2 \, , z_2)$ of singular vectors associated to $\tau_2$.
	
	Assume that $\tau_2$ is a singular ordered value of $T$, and $x_2$ , $ y_2 $ and $z_2$ are singular ordered vectors associated to $\tau_2$, that is
	\[
	\left \{ \begin{array}{rc}
	T(x,y_2)=\tau_2 \langle x,x_2\rangle  z_2 &   \\
	T(x_2,y)=\tau_2 \langle y,y_2\rangle z_2 &   \\
	\phi^*_{H_1}(y)(z_2)=\tau_2 \langle y,y_2\rangle x_2 &
	\end{array}
	\right . 
	\]
	for all $(x,y) \in H_1 \times H_2$. Let $T_3 \in Bil(H_1 \times H_2,K)$ given by 
	\[
	T_3(x,y)=T(x,y)-\displaystyle\sum^2_{i=1}\tau_i \langle x,x_i\rangle \langle y,y_i\rangle z_i\, ,
	\]
	for all $(x,y) \in H_1  \times H_2$. Then, $T_3$ is a compact bilinear operator. If $T_3$ is nonzero, $\tau_3 = \| T_3\|$ is a singular value of $T_3$ with a triple $(x_3\, , y_3 \, , z_3)$ of singular vectors associated to $\tau_3$. We have
	\[
	\left \{ \begin{array}{rc}
	T_3(x_3,y_3)  = \tau_3 z_3 \\
	\phi^*_{3H_1}(y_3)(z_3)  = \tau_3 x_3 \\
	\phi^*_{3H_2}(x_3)(z_3)  = \tau_3 y_3 \, .
	\end{array}\right .
	\]
	
	We claim that $x_3 \perp x_i$, $y_3 \perp y_i$, $z_3 \perp
	z_i$, $i=1,2$ and $\tau_2 \geq \tau_3$. Indeed,
	\[
	\begin{array}{rcl}
	\tau_3 \langle x_i,x_3\rangle &=& \langle x_i,\tau_3 x_3\rangle = \langle x_i,\phi^*_{3H_1}(y_3)(z_3)\rangle \\
	&=& \langle \phi_{3H_1}(y_3)(x_i),z_3\rangle = \langle T_3(x_i,y_3), z_3\rangle \\
	&=& \langle T(x_i,y_3)-\displaystyle\sum^2_{k=1}\tau_k \langle x_i,x_k\rangle \langle y_3,y_k\rangle z_k,z_3\rangle \\
	&=& \langle T(x_i,y_3)-\tau_i \langle y_3,y_i\rangle z_i,z_3\rangle \\
	&=& \langle \tau_i \langle y_3,y_i\rangle z_i-\tau_i \langle y_3,y_i\rangle z_i,z_3\rangle \\
	&=&0\, ,\\
	\end{array}
	\]
	\[
	\begin{array}{rcl}
	\tau_3 \langle y_i,y_3 \rangle &=& \langle y_i,\tau_3 y_3\rangle = \langle y_i,\phi^*_{3H_2}(x_3)(z_3)\rangle \\
	&=& \langle \phi_{3H_2}(x_3)(y_i),z_3\rangle = \langle T_3(x_3,y_i),z_3\rangle \\
	&=& \langle T(x_3,y_i)-\displaystyle\sum^2_{k=1}\tau_k \langle x_3,x_k\rangle \langle y_i,y_k\rangle z_k,z_3\rangle  \\
	&=& \langle T(x_3,y_i)-\tau_i \langle x_3,x_i\rangle z_i,z_3\rangle  \\
	&=& \langle \tau_i \langle x_3,x_i\rangle z_i-\tau_i \langle x_3,x_i\rangle z_i,z_3\rangle  \\
	&=&0\, , \\
	\end{array}
	\]
	and
	\[
	\begin{array}{rcl}
	\tau_3 \langle z_i,z_3\rangle &=& \langle z_i,\tau_3 z_3\rangle = \langle z_i,T_3(x_3,y_3)\rangle  \\
	&=& \langle z_i,T(x_3,y_3)\rangle = \langle z_i,\phi_{H_1}(y_3)(x_3)\rangle  \\
	&=& \langle \phi^*_{H_1}(y_3)(z_i),x_3\rangle = \langle \tau_i \langle y_3,y_i\rangle x_i,x_3\rangle  \\
	&=&0 \, . \\
	\end{array}
	\]
	
	We have
	\[
	\tau_3=\| T_3\|=\sup\{\|T_3(x,y)\|_{_{K}}: \mbox{
	where } \|x\|_{_{H_1}}=\|y\|_{_{H_2}}=1, x\in H_1, y\in
	H_2\}
	\]
	\[
	\tau_2=\| T_2\|=\sup\{\|T_2(x,y)\|_{_{K}}: \mbox{ where } \|x\|_{_{H_1}}=\|y\|_{_{H_2}}=1, x\in
	H_1, y\in H_2\}\, .
	\]
	We also have that $T_3(x,y) \perp z_2$, for all $(x,y) \in H_1 \times H_2$. Indeed, 
	\[
	\begin{array}{rcl}
	\langle T_3(x,y),z_2\rangle &=& \langle T(x,y)-\displaystyle\sum^2_{i=1}\tau_i \langle  x,x_i\rangle \langle y,y_i\rangle z_i,z_2\rangle \\
	&=& \langle T(x,y),z_2\rangle -\tau_2 \langle x,x_2\rangle \langle y,y_2\rangle \\
	&=& \langle \phi_{H_1}(y)(x),z_2 \rangle -\tau_2 \langle x,x_2\rangle \langle y,y_2\rangle  \\
	&=& \langle x,\phi^*_{H_1}(y)(z_2)\rangle -\tau_2 \langle x,x_2\rangle \langle y,y_2\rangle  \\
	&=& \langle x,\tau_2 \langle y,y_2\rangle x_2\rangle - \langle x,\tau_2 \langle y,y_2\rangle x_2\rangle  \\
	&=&0\, . \\
	\end{array}
	\]
	So, $T_3(x,y) \perp \tau_2 \langle x,x_2\rangle \langle y,y_2\rangle z_2$, for all $(x,y) \in H_1 \times H_2$. Thus, 
	\[
	\|T_3(x,y)\|_{_{K}}^2+\|\tau_2 \langle x,x_2\rangle \langle y,y_2\rangle z_2\|_{_{K}}^2=\|T_2(x,y)\|_{_{K}}^2
	\]
	for all $(x,y) \in H_1\times H_2$, giving that $\tau_2 \geq \tau_3$.
	
	Since $x_3\perp x_i$,  $y_3\perp y_i$ and $z_3\perp z_i$, $i=1,2$, we get that $T_3(x_3,y_3)=T(x_3,y_3)$, implying $T(x_3,y_3)=\tau_3 z_3$.
	
	Note that
	\[
	\phi_{3H_1}(y)(x)=T_3(x,y)=T(x,y)-\displaystyle\sum^2_{i=1}\tau_i \langle x,x_i\rangle \langle y,y_i\rangle z_i \, ,
	\]
	then,
	\[
	\phi_{3H_1}(y)(x)=\phi_{H_1}(y)(x)-\displaystyle\sum^2_{i=1}\tau_i \langle x,x_i\rangle \langle y,y_i\rangle z_i
	\]
	For all $(x,y) \in H_1\times H_2$. Let $z \in K$, we have
	\[
	\begin{array}{rcl}
	\langle \phi_{3H_1}(y)(x),z\rangle &=& \langle \phi_{H_1}(y)(x)-\displaystyle\sum^2_{i=1}\tau_i \langle x,x_i\rangle \langle y,y_i\langle z_i,z\rangle  \\
	&=& \langle \phi_{H_1}(y)(x),z\rangle - \langle \displaystyle\sum^2_{i=1}\tau_i \langle x,x_i\rangle \langle y,y_i\rangle z_i,z\rangle \\
	&=& \langle x,\phi^*_{H_1}(y)(z)\rangle - \langle x,\displaystyle\sum^2_{i=1}\tau_i \langle y,y_i\rangle \langle z,z_i\rangle x_i\rangle\, .
	\end{array}
	\]
	Thus,
	\[
	\phi^*_{3H_1}(y)(z)=\phi^*_{H_1}(y)(z)-\displaystyle\sum^2_{i=1}\tau_i \langle y,y_i\rangle \langle z,z_i\rangle x_i\,.
	\]
	
	In a similar way, we get 
	\[
	\phi^*_{3H_2}(x)(z)=\phi^*_{H_2}(x)(z)-\displaystyle\sum^2_{i=1}\tau_i \langle x,x_i\rangle \langle z,z_i\rangle y_i\, .
	\]
	So, 
	\[
	\phi^*_{H_1}(y_3)(z_3)=\phi^*_{3H_1}(y_3)(z_3)=\tau_3 x_3
	\]
	\[
	\phi^*_{H_2}(x_3)(z_3)=\phi^*_{3H_2}(x_3)(z_3)=\tau_3 y_3\, .
	\]
	Therefore, we obtain that $\tau_3$ is a singular value of de $T$ with a triple $(x_3\, , y_3 \, , z_3)$ of singular vectors associated to $\tau_3$.
	
	Assume that $\tau_3$ is an ordered singular value of $T$ and $x_3$, $y_3$ and $z_3$ are ordered singular vectors associated to $\tau_3$, that is
	\[
	\left \{ \begin{array}{rc}
	T(x,y_3)=\tau_3 \langle x,x_3\rangle z_3 &   \\
	T(x_3,y)=\tau_3 \langle y,y_3\rangle z_3 &   \\
	\phi^*_{H_1}(y)(z_3)=\tau_3 \langle y,y_3\rangle x_3 &
	\end{array}
	\right .
	\]
	for all $(x,y) \in H_1 \times H_2$.
	
	This process continues and suppose that up to step $n$, $n \in \Nset$, we have that $\tau_i$ is an ordered singular value of $T$ with ordered vectors $x_i$, $y_i$, $z_i$, respectively, where $ i = 1,2, \ldots,n$, such that $\|T\| = \tau_1 \geq
	\tau_2 \geq \ldots \geq \tau_n$, $x_i \perp x_j$, $y_i \perp y_j $,
	$z_i \perp z_j$ always that $i\neq j$,
	\[
	\left \{ \begin{array}{rc}
	T(x_i,y_i)  = \tau_i z_i \\
	\phi^*_{H_1}(y_i)(z_i)  = \tau_i x_i \\
	\phi^*_{H_2}(x_i)(z_i)  = \tau_i y_i
	\end{array}
	\right .
	\]
	and
	\[
	\left \{ \begin{array}{rc}
	T(x,y_i)=\tau_i \langle x,x_i\rangle z_i &   \\
	T(x_i,y)=\tau_i \langle y,y_i\rangle z_i &   \\
	\phi^*_{H_1}(y)(z_i) = \tau_i \langle y,y_i\rangle x_i &
	\end{array}
	\right .
	\]
	for all $(x,y) \in H_1 \times H_2$.
	
	We also have
	\[
	\left \{ \begin{array}{rc}
	\phi_{nH_1}(y)(x)=\phi_{H_1}(y)(x)-\displaystyle\sum^{n-1}_{i=1}\tau_i \langle x,x_i\rangle \langle y,y_i\rangle z_i &   \\
	\phi_{nH_2}(x)(y)=\phi_{H_2}(x)(y)-\displaystyle\sum^{n-1}_{i=1}\tau_i \langle x,x_i\rangle \langle y,y_i\rangle z_i &  \\
	\phi^*_{nH_1}(y)(z)=\phi^*_{H_1}(y)(z)-\displaystyle\sum^{n-1}_{i=1}\tau_i \langle y,y_i\rangle \langle z,z_i\rangle x_i&  \\
	\phi^*_{nH_2}(x)(z)=\phi^*_{H_2}(x)(z)-\displaystyle\sum^{n-1}_{i=1}\tau_i \langle x,x_i\rangle \langle z,z_i\rangle y_i& \\
	T_n(x,y)=T(x,y)-\displaystyle\sum^{n-1}_{i=1}\tau_i \langle x,x_i\rangle \langle y,y_i\rangle z_i & \mbox{ e } \tau_i=\|T_i\|, \, \, \mbox{i=2,3,...,n}
	\end{array}
	\right . \, .
	\]
	
	If the process stops at any stage, say step $m$, then the operator $T_{m + 1}$ is null, and we have
	\[
	T(x,y)=\displaystyle\sum^{m}_{i=1}\tau_i \langle x,x_i\rangle \langle y,y_i\rangle z_i
	\]
that is, $T$ is a finite rank operator. If we put $\tau_i = 0$, $x_i = y_i = z_i = 0$ for $ i> m $, we have
	\[
	T(x,y)=\displaystyle\sum^{\infty}_{i=1}\tau_i \langle x,x_i\rangle \langle y,y_i\rangle z_i\, .
	\]
	
	Otherwise, we have $\tau_n = \|T_n \|$, and we claim that $\tau_n \rightarrow 0$ when $ n \rightarrow \infty$ , that is,  $\|T_n\| \rightarrow 0$, and then
	\[
	T(x,y)=\displaystyle\sum^{\infty}_{i=1}\tau_i \langle x,x_i\rangle \langle y,y_i\rangle z_i\, .
	\]
	Indeed, we have that $T$ is a compact operator. Let $T(x_i,y_i)=\tau_i z_i $ and $T(x_j,y_j) = \tau_j z_j,$ where $i\neq j,$ $i,j=1,2,\ldots, n$. Then,
	\[
	\|T(x_i,y_i) - T(x_j,y_j)\|_{_{K}} ^2 = \|\tau_i z_i -\tau_j z_j\|_{_{K}}^2 = \langle \tau_i z_i-\tau_j z_j,\tau_i z_i-\tau_j z_j\rangle = \tau_i^2 + \tau_j^2
	\]
	\[
	\|T(x_i,y_i) - T(x_j,y_j)\|_{_{K}} ^2 = \tau_i^2 + \tau_j^2\, .
	\]
	Since $(x_n,y_n)$ is bounded, $\|x_n\|_{_{H_1}}=\|y_n\|_{_{H_2}}=1$, then $(T(x_n,y_n))$ has a convergent subsequence, $(T(x_{n_k},y_{n_k}))$. So, $(T(x_{n_k},y_{n_k}))$ is a Cauchy sequence, that is, $\|T(x_{n_i},y_{n_i}) - T(x_{n_j},y_{n_j})\|_{_{K}} ^2 \rightarrow 0$. Thus, $\tau_n \rightarrow 0$. This completes the proof.
\end{proof}

\begin{remark}
	Let $(\tau_n)$ be the ordered singular values obtained in Theorem \ref{existenciadeschmidt}. We have
	\[
	\left \{ \begin{array}{rc}
	\phi_{nH_1}(y_n)(x)=\phi_{H_1}(y_n)(x) &   \\
	\phi_{nH_2}(x_n)(y)=\phi_{H_2}(x_n)(y) &  \\
	\phi^*_{nH_1}(y_n)(z)=\phi^*_{H_1}(y_n)(z) &  \\
	\phi^*_{nH_2}(x_n)(z)=\phi^*_{H_2}(x_n)(z) & \\
	\end{array}
	\right .
	\]
	for all $ (x,y)\in H_1 \times H_2$ e $z\in K$. Thus,
	\[\left \{ \begin{array}{rc}
	\phi^*_{H_1}(y_n) \circ  \phi_{H_1}(y_n)(x_n)=\tau^2_n x_n  & \mbox{ then $\tau^2_n$ is eigenvalue of  } \phi^*_{H_1}(y_n) \circ  \phi_{H_1}(y_n): H_1 \rightarrow H_1 \\
	\phi^*_{H_2}(x_n) \circ  \phi_{H_2}(x_n)(y_n)=\tau^2_n y_n  & \mbox{ then $\tau^2_n$ is eigenvalue of } \phi^*_{H_2}(x_n) \circ  \phi_{H_2}(x_n): H_2 \rightarrow H_2 \\
	\phi_{H_1}(y_n) \circ  \phi^*_{H_1}(y_n)(z_n)=\tau^2_n z_n  & \mbox{ then $\tau^2_n$ is eigenvalue of } \phi_{H_1}(y_n) \circ  \phi^*_{H_1}(y_n): K \rightarrow K \\
	\phi_{H_2}(x_n) \circ  \phi^*_{H_2}(x_n)(z_n)=\tau^2_n z_n  & \mbox{ then $\tau^2_n$ is eigenvalue of } \phi_{H_2}(x_n) \circ  \phi^*_{H_2}(x_n): K \rightarrow K \, .\\
	\end{array}
	\right .
	\]
\end{remark}

The next result give us that the conditions of Theorem 3 is fundamental.

\begin{theorem}
	Suppose that $T\in Bil(H_1\times H_2,K)$ has a monotone Schmidt representation given by 
	\[
	T(x,y):=\displaystyle\sum^{\infty}_{i=1}\tau_i \langle x,x_i\rangle \langle y,y_i\rangle z_i \,,
	\]
	where $(\tau_i)\in c_0$ and $(x_i)$, $(y_i)$ e $(z_i)$ are orthonormal sets of $H_1$, $H_2$ and $K$, respectively. We have that $T\in Bil(H_1\times H_2,K)$ is compact and the sequence  $(\tau_n)_{_{n\in \Nset}}$ is a sequence of ordered singular values of $T$.
\end{theorem}

\begin{proof}
	
	For each $N \in \N$, if we define $T_N(x,y)  = \displaystyle \sum^N_{i=1}\tau_i \langle x,x_i\rangle \langle y,y_i\rangle z_i$, then each $T_N$ is compact, and $T_N$ converges in norm to $T$, giving the compactness of $T$. Now, we have that
	\[
	\begin{array}{rcl}
	\langle \phi_{H_1}(y)(x),z\rangle&=&\langle T(x,y),z\rangle=\langle \sum^{\infty}_{i=1}\tau_i\langle x,x_i\rangle \langle y,y_i\rangle z_i,z\rangle
	\\
	&=&\sum^{\infty}_{i=1}\tau_i \langle x,x_i\rangle \langle y,y_i\rangle \langle z,z_i\rangle\\
	&=&\langle x,\sum^{\infty}_{i=1}\tau_i \langle y,y_i\rangle \langle z,z_i\rangle x_i \rangle \,
	.\\
	\end{array}
	\]
	Thus, $\phi^*_{H_1}(y)(z)=\displaystyle\sum^{\infty}_{i=1}\tau_i \langle y,y_i\rangle \langle z,z_i\rangle x_i$. In a similar way we show that $\phi^*_{H_2}(x)(z)=\displaystyle
	\sum^{\infty}_{i=1}\tau_i \langle x,x_i\rangle \langle z,z_i\rangle y_i$. Then, we have
	\[
	\left \{ \begin{array}{rc}
	T(x_i,y_i)=\tau_i z_i &   \\
	\phi^*_{H_1}(y_i)(z_i)=\tau_i x_i &  \\
	\phi^*_{H_2}(x_i)(z_i)=\tau_i y_i &  \\
	T(x,y_i)=\tau_i \langle x,x_i\rangle z_i & \\
	T(x_i,y)=\tau_i \langle y,y_i\rangle z_i & \\
	\phi^*_{H_1}(y)(z_i)=\tau_i \langle y,y_i\rangle x_i \, .& \\
	\end{array}
	\right . 
	\]
	Therefore, $\tau_i$ is an ordered singular value of $T$ with associated ordered singular vectors $(x_i)$, $(y_i)$ e $(z_i)$.
\end{proof}

\begin{example}
	If $x = (x_1,x_2,x_3), y = (y_1,y_2,y_3) \in \Rset^3$, let $T \in Bil(\Rset^3)$ be given by 

	\begin{align*}
	T(x,y) &= (\frac{\sqrt{6}}{6}x_1 (y_1 + y_3) - \sqrt{2}x_2 y_2 + \frac{\sqrt{3}}{2}x_3 (y_1 + y_3),\\
	& \frac{\sqrt{6}}{6}x_1 (y_1 + y_3) + \sqrt{2}x_2 y_2 + \frac{\sqrt{3}}{2}x_3 (y_1 - y_3),\\
	& \frac{\sqrt{6}}{6}x_1 (y_1 + y_3) - \sqrt{3} x_3 (y_1 - y_3)) \,.
	\end{align*}
	
	We obtain the following singular values with their respective associated singular vectors:\\
	\centerline{$\tau_1=1$ with $x_1=(1, \, 0, \, 0)$, $\displaystyle
		y_1=(\frac{1}{\sqrt{2}}, \, 0,\, \frac{1}{\sqrt{2}})$ and
		$\displaystyle
		z_1=(\frac{1}{\sqrt{3}}, \, \frac{1}{\sqrt{3}},\, \frac{1}{\sqrt{3}})$}\\
	\centerline{$\tau_2=2$ with $x_2=(0, \, 1, \, 0)$, $\displaystyle
		y_2=(0, \, 1,\, 0)$ and  $\displaystyle
		z_2=(-\frac{1}{\sqrt{2}}, \, \frac{1}{\sqrt{2}},\, 0)$}\\
	\centerline{$\tau_3=3$ with $x_3=(0, \, 0, \, 1)$, $\displaystyle
		y_3=(\frac{2}{\sqrt{8}}, \, 0,\, -\frac{2}{\sqrt{8}})$ and
		$\displaystyle
		z_3=(\frac{1}{\sqrt{6}}, \, \frac{1}{\sqrt{6}},\, -\frac{2}{\sqrt{6}})$\,.}\\
	
	After some calculations it can be shown that $\tau_i$, $i = 1,2,3 $, are ordered regular values with their ordered singular vectors associated to $x_i, \, y_i, \, z_i$. It is easy to show that $ x_i \perp x_j$, $y_i \perp y_j$ and $z_i \perp z_j $, where $i, j = 1,2,3$ and $i \neq j$.
\end{example}

\section{The Schur representation}

In \cite{EDM} is defined the Schur representation of a compact bilinear operator $T \in Bil(H)$, and it is given new definitions, results and conditions in order to obtain a spectral theorem for bilinear operators on real Hilbert spaces.

\begin{definition}\label{shurdefini} Let $H$ be a separable Hilbert space and $T:H\times H \rightarrow H$ be a bilinear operator. $T$ has a Schur representation if there exists a sequence $(\lambda_n) \in c_0$ and an extended orthonormal sequence $(x_n)$ in $H$, such that
	\[
	T(x,y)=\sum^\infty_{n=1}\lambda_n\langle x,x_n\rangle \langle y,x_n\rangle x_n \,,
	\]
	for all $x, y \in H$.
\end{definition}
A Schur representation is monotone if $|\lambda_1| \geq |\lambda_2| \geq \ldots \geq 0$.

For each $x \in H$, and $T\in Bil(H)$, we define $T^s_x \in \mathfrak{L}(H)$ by $T^s_x(y) = T(y,x)$. In general $T_x \neq T^s_x$, but if $T$ is symmetric, then $T_x = T_x^s$, for all $x \in H$.

\begin{definition}\label{adj04}
	$T \in {\rm Bil}(H)$ is self-adjoint, if $T_x$ and $T_x^s$ are self-adjoint linear operators for all $x \in H$.
\end{definition}

In  \cite{JRR} it is shown that, if $L \in \mathfrak{L}(H)$ is compact, positive and self-adjoining, then the Schur and Schmidt representations of $L$ coincide, that is, the singular values are eigenvalues. We will show an analogous result for the bilinear case.

Let $T: H \times H \rightarrow H$ be compact. We know that $\tau_1 = \| T \|$ is a singular value of $T$, that is, there are vectors $x_1$, $y_1$ and $z_1 \in H$, such that,
\[
\left \{ \begin{array}{rc}
T(x_1,y_1)=\tau_1 z_1 &   \\
\phi^*_{H_1}(y_1)(z_1)=\tau_1 x_1 &  \\
\phi^*_{H_2}(x_1)(z_1)=\tau_1 y_1 &  \\
\end{array}
\right . \, .
\]
Assume that $T$ is self-adjoint, that is,\\
\centerline{$\langle T(x,y),z\rangle =\langle y,T(x,z)\rangle$}\\
\centerline{$\langle T(y,x),z\rangle = \langle y,T(z,x)\rangle$}\\
for all $x,\, y, \, z \in H$. In \cite{EDM} it is proved that in this situation we get that $T$ is symmetric. Then, \\
\centerline{$\langle \phi_{H_1}(y)(x),z\rangle = \langle T(x,y),z\rangle = \langle x,T(z,y)\rangle = \langle x,\phi_{H_1}(y)(z)\rangle$}\\
\centerline{$\langle \phi_{H_2}(x)(y),z\rangle = \langle T(x,y),z\rangle = \langle y,T(x,z)\rangle = \langle y,\phi_{H_2}(x)(z)\rangle$.}\\

Thus, $\phi^*_{H_1}(y)=\phi_{H_1}(y)$  and $\phi^*_{H_2}(x)=\phi_{H_2}(x)$, that is, $\phi_{H_1}(y)$ and $\phi_{H_2}(x)$ are self-adjoints. So, we have 
\begin{eqnarray}
T(x_1,y_1)&=&\tau_1 z_1 \nonumber   \\
T(z_1,y_1)&=&\tau_1 x_1  \label{propvalorsingularschmidt}  \\
T(x_1,z_1)&=&\tau_1 y_1 \nonumber \, .
\end{eqnarray}
Thus, if $T\in Bil(H)$ is compact and self-adjoint, then $\tau_1=\|T\|$ is a singular value with associated unitary singular vectors $x_1, \, y_1, \, z_1 \in H$, such that \eqref{propvalorsingularschmidt} is valid. If $\tau_1$ is an ordered singular value with ordered singular vectors $x_1$, $y_1$ and $z_1$ associated to $\tau_1$, we have:
\[
\begin{array}{rcl}
T(x,y_1)&=&\tau_1 \langle x,x_1\rangle z_1 \\
T(x_1,y)&=&\tau_1 \langle y,y_1\rangle z_1  \\
T(z_1,y)&=&\tau_1 \langle y,y_1\rangle x_1 \, ,\\
\end{array}
\]
for all $x, \, y \in H$. Note that \\
\centerline{$T(y_1,y_1)= \tau_1 \langle y_1,x_1\rangle z_1$}\\
\centerline{$T(x_1,x_1)= \tau_1 \langle x_1,y_1\rangle z_1\, ,$}
that is, $T(y_1,y_1)=T(x_1,x_1)$ and $T(z_1,z_1)=\tau_1 \langle _1,y_1\rangle x_1$.

Now, let us prove that if $T\in Bil(H)$ is compact, self-adjoint with a Schmidt representation, then it has a Schur representation of $T$, that is, the ordered singular values of $T$ that appear in his Schmidt representation are also ordered eigenvalues.
\begin{theorem}
	Let $T \in Bil(H)$ be compact, self-adjoining and not null. If $T$ has a Schmidt Representation according to Theorem \ref{existenciadeschmidt}, that is $T(x, y) = \displaystyle \sum^\infty_{j = 1} \tau_j \langle x, x_j\rangle \langle y, y_j\rangle z_j$, then $T$ has a Schur representation, that is
	\[
	T(x,y)=\displaystyle \sum^\infty_{j=1}\lambda_j \langle x,x_j\rangle \langle y,x_j\rangle x_j
	\]
	for all $x, y \in H$.
\end{theorem}
\begin{proof} Let $T(x,y)=\displaystyle \sum^\infty_{j=1}	\tau_j \langle x,x_j\rangle \langle y,y_j\rangle z_j$ be a Schmidt representation of $T$ according to Theorem \ref{existenciadeschmidt}. We have \\
	\centerline{$T(x_i,y_i)=\tau_i z_i  \hspace{1cm}\mbox{ \,
		}\hspace{1cm}T(x,y_i) = \tau_i \langle x,x_i\rangle z_i$}\\
	\centerline{$T(z_i,y_i)=\tau_i x_i  \hspace{1cm}\mbox{ \,
		}\hspace{1cm}T(x_i,y)=\tau_i \langle y,y_i\rangle z_i$}\\
	\centerline{$T(x_i,z_i)=\tau_i y_i  \hspace{1cm}\mbox{ \,
		}\hspace{1cm}T(z_i,y)=\tau_i \langle y,y_i \rangle x_i$}\\
	for all $x, \, y \in H$.
	
	We have that $T(y_i,x_i)=T(x_i,y_i)=\tau_i z_i$, so\\
	\centerline{$T(y_i,x_i)=\displaystyle \sum^\infty_{j=1}
		\tau_j \langle y_i,x_j\rangle \langle x_i,y_j\rangle z_j = \tau_i z_i$\, .}\\
	Then,\\
	\centerline{$\tau_i \langle y_i,x_i\rangle \langle x_i,y_i\rangle \langle z_i,z_i\rangle = \tau_i \langle z_i,z_i\rangle$\,.}\\
	Since $\tau_i \neq 0$ and $z_i$ is unitary, we get 
	\begin{equation}\label{1}
	\langle y_i,x_i\rangle = \pm 1\, .
	\end{equation}
	We also have that $T(z_i,y_i)=T(y_i,z_i)=\tau_i x_i$, so\\
	\centerline{$T(z_i,y_i)=\displaystyle \sum^\infty_{j=1}
		\tau_j \langle z_i,x_j\rangle \langle y_i,y_j\rangle z_j = \tau_i x_i$}\\
	\centerline{$T(y_i,z_i)=\displaystyle \sum^\infty_{j=1}
		\tau_j \langle y_i,x_j\rangle \langle z_i,y_j\rangle z_j = \tau_i x_i$\, .}\\
	Thus,\\
	\centerline{$\tau_i \langle z_i,x_i\rangle \langle y_i,y_i\rangle z_i = \tau_i x_i$}\\
	\centerline{$\tau_i \langle y_i,x_i\rangle \langle z_i,y_i\rangle \langle z_i,z_i\rangle = \tau_i \langle x_i,z_i>\rangle$\, .}\\
	Since $\tau_i > 0$ and $y_i$ is unitary, 
	\begin{eqnarray}
	\langle z_i,x_i\rangle z_i&=& x_i\label{2}\\
	\langle y_i,x_i\rangle \langle z_i,y_i\rangle &=& \langle x_i,z_i\rangle\, .\label{3}
	\end{eqnarray}
	We also have that $T(x_i,z_i)=T(z_i,x_i)=\tau_i y_i$, so \vspace{0.3cm}\\
	\centerline{$T(x_i,z_i)=\displaystyle \sum^\infty_{j=1}
		\tau_j \langle x_i,x_j\rangle \langle z_i,y_j\rangle z_j=\tau_i y_i$}\\
	\centerline{$T(z_i,x_i)=\displaystyle \sum^\infty_{j=1}
		\tau_j \langle z_i,x_j\rangle \langle x_i,y_j\rangle z_j = \tau_i y_i$\, .}\\
	This implies,\\
	\centerline{$\tau_i \langle z_i,y_i\rangle z_i = \tau_i y_i$}\\
	\centerline{$\tau_i \langle z_i,x_i\rangle \langle x_i,y_i\rangle \langle z_i,z_i\rangle = \tau_i \langle y_i,z_i\rangle$\, .}\\
	
	Since $\tau_i > 0$ and $z_i$ is unitary, we get
	\begin{eqnarray}
	\langle z_i,y_i\rangle z_i&=& y_i\label{4}\\
	\langle z_i,x_i\rangle \langle x_i,y_i\rangle &=& \langle y_i,z_i\rangle\, .\label{5}
	\end{eqnarray}
	Provided that $\tau_i$ is an ordered singular value, we also have:
	$$T(y_i,y_i)=\tau_i \langle y_i,x_i\rangle z_i$$
	$$T(x_i,x_i)=\tau_i \langle x_i,y_i\rangle z_i$$
	$$T(z_i,z_i)=\tau_i \langle z_i,y_i\rangle x_i.$$
	
	By \eqref{1}, $\langle y_i,x_i\rangle =\pm 1 $. Let $\langle y_i,x_i\rangle = 1$. Then, by \eqref{3}, $\langle z_i,y_i\rangle = \langle x_i,z_i\rangle$. From \eqref{2} and \eqref{4}, we get that $y_i=x_i$.
	
	By \eqref{2}, we have $\langle z_i,x_i\rangle z_i = x_i$, so $\langle z_i,x_i\rangle = \pm 1$. Let $\langle z_i,x_i\rangle = 1$. Since
	\begin{eqnarray}
	T(z_i,y_i)&=&\tau_i x_i\label{6}\\
	T(x,y_i)&=&\tau_i \langle x,x_i\rangle z_i\, ,\label{7}
	\end{eqnarray}
	for all $x\in H$, by \eqref{7} we have \\
	\centerline{$T(z_i,y_i)=\tau_i \langle z_i,x_i\rangle z_i = \tau_i z_i.$}\\
	For $\tau_i > 0$, by \eqref{6} we have $z_i=x_i$. If we put $\lambda_i=\tau_i$, $i\in \Nset$, then
	\[
	T(x,y)=\displaystyle \sum^\infty_{i=1}
	\lambda_i \langle x,x_i\rangle \langle y,x_i\rangle x_i\, ,
	\]		
	is a Schur representation of $T$. If $\langle z_i,x_i\rangle = -1$, by \eqref{7} we have $T(z_i,y_i)=-\tau_i z_i$. Then, by \eqref{6} we get $z_i=-x_i$. If we put $\lambda_i = -\tau_i$, $i\in \Nset$, then 
	\[
	T(x,y)=\displaystyle \sum^\infty_{i=1}
	\lambda_i \langle x,x_i\rangle \langle y,x_i\rangle x_i\, ,
	\]		
	is a Schur representation of $T$. 
	
	Suppose now that $\langle y_i,x_i\rangle = -1$. Then, by \eqref{3} we have $\langle z_i,y_i\rangle = -\langle z_i,x_i\rangle$. Thus, by \eqref{2} and \eqref{4} we get that $y_i=-x_i$. If $\langle z_i,x_i\rangle = 1$, we have $z_i = x_i$. If we put $\lambda_i = -\tau_i$, $i\in \Nset$, then 
	\[
	T(x,y)=\displaystyle \sum^\infty_{i=1}
	\lambda_i \langle x,x_i\rangle \langle y,x_i\rangle x_i\, ,
	\]
	is a Schur representation of $T$.
	
	If $\langle z_i,x_i\rangle = -1$, then $z_i=-x_i$. If we put $\lambda_i=\tau_i$, $i\in \Nset$, we have 
	\[T(x,y)=\displaystyle \sum^\infty_{i=1}
	\lambda_i \langle x,x_i\rangle \langle y,x_i\rangle x_i\, ,
	\]		
	is a Schur representation of $T$.
\end{proof}

\section{Schmidt representation and Hilbert-Schmidt bilinear operators}

We know that if $L: H_1 \rightarrow H_2$ is a linear compact operator, then $L$ has a Schmidt representation, that is, $L(x)=\displaystyle \sum_{n=1}^\infty \tau_n<x,x_n>y_n$. In \cite{MM} we have the following definition:

\begin{definition}
	Let $H_1$, $H_2$ and $K$ be separable Hilbert spaces. An operador $T\in Bil(H_1\times H_2,K)$ is a bilinear Hilbert-Schmidt operator if the series 
	\[
	||T||_2=\left(\sum_{m=1}^\infty\sum_{n=1}^\infty ||T(u_m,v_n)||_K^2\right)^{1/2}
	\]
	converges for some orthonormal basis $(u_m)$ of $H_1$ and $(v_n)$ de $H_2$.
\end{definition}

The class of all bilinear Hilbert-Schmidt operators $T\in Bil(H_1\times H_2,K)$ is denoted by $HS(H_1\times H_2,K)$. When $K$ is the scalar field we write $HS(H_1\times H_2)$.\\

It can be proved that $HS(H_1\times H_2,K)$ is a Hilbert space with the norm $\| \cdot \|_{HS}$ defined from the inner product \\
\centerline{ $\displaystyle \langle T,S \rangle =\sum^\infty_{i,j=1}
	\langle T(x_i,y_j),S(x_i,y_j)\rangle  \,
	.$}\\

The next theorems from \cite{MM} give us that Definition 3.1 is independent of the orthonormal bases in $H_1$ and $H_2$. 

\begin{theorem}\label{prophilbertbase}
	Let  $T\in HS(H_1\times H_2,K)$, then $||T||_2$ is independent of choice of the orthonormal bases. 
\end{theorem}

\begin{proposition}
	Let $T \in HS(H_1\times H_2,K)$, then 
	\[
	||T||_{Bil(H_1\times H_2,K)} \leq ||T||_{HS(H_1\times H_2,K)} \,.
	\]
\end{proposition}

\begin{theorem}
	Let $T \in HS(H_1\times H_2,K)$, then $T$ is compact.
\end{theorem}

Now, we get the following result. 

\begin{theorem} Let $H_1$, $H_2$ and $K$ Hilbert spaces, and let $T\in Bil(H_1\times
	H_2,K)$ be compact and not null. Suppose that $T$ has a Schmidt representation as in the Theorem \ref{existenciadeschmidt}. Let $(\tau_n)$ be the sequence of singular values of $T$ given in the Schmidt representation of $T$. If $(\tau_n) \in \ell^2$, then $T$ is a 	Hilbert-Schmidt operator.
\end{theorem}

\begin{proof} Let $T(x,y)=\displaystyle\sum^{\infty}_{i=1}\tau_i \langle x,x_i\rangle \langle y,y_i\rangle z_i$ and $(\tau_i) \in \ell^2$, that is, $\displaystyle \sum_{i=1}^\infty |\tau_i|^2 < \infty$. Suppose that $(u_i)$ and $(v_i)$, $i\in \Nset$, are orthonormal bases of $H_1$ and $H_2$,
	respectively. We want to show that $\displaystyle	\sum_{i,j=1}^\infty \|T(u_i,v_j) \|_{_{K}}^2 < \infty$. We have 
	\[
	\|x_s\|_{_{H_1}}=\displaystyle \sum_{i=1}^\infty |\langle u_i,x_s\rangle |^2=1\, \mbox{
		e } \|y_s\|_{_{H_2}}=\displaystyle \sum_{j=1}^\infty
	|\langle v_j,y_s\rangle |^2=1.
	\]
	Then, 
	\begin{eqnarray*}
		\sum_{i,j=1}^\infty \|T(u_i,v_j) \|_{_{K}}^2 &=&\sum_{i,j=1}^\infty
		\left (\|\sum^{\infty}_{s=1}\tau_s \langle u_i,x_s\rangle \langle v_j,y_s\rangle z_s \|_{_{K}}^2\right)\\
		&=& \sum_{i,j=1}^\infty \left
		(\sum^{\infty}_{s=1}|\tau_s|^2|\langle u_i,x_s\rangle |^2|\langle v_j,y_s\rangle |^2\right)\\
		&=&\sum_{s=1}^\infty \left (\sum^{\infty}_{i=1}\left
		(\sum^{\infty}_{j=1}|\tau_s|^2|\langle u_i,x_s\rangle |^2|\langle v_j,y_s\rangle |^2\right)\right)\\
		&=&\sum_{s=1}^\infty \left (\sum^{\infty}_{i=1}|\tau_s|^2|\langle u_i,x_s\rangle |^2\right)\\
		&=&\sum_{s=1}^\infty |\tau_s|^2 < \infty \, ,
	\end{eqnarray*}
	as desired.
\end{proof}

\vspace{0.5cm}

\noindent

Eduardo Brandani da Silva  \\
Universidade Estadual de Maring\'a--UEM \\
Departamento de Matem\'atica \\
Av. Colombo 5790 \\
Maring\'a - PR  \\
Brazil - 870300-110

\vspace{1 mm}

\noindent E-mail: \,{\tt ebsilva@wnet.com.br}

\vspace{0.3cm}

Dicesar Lass Fernandez  \\
Universidade Estadual de Campinas - Unicamp \\
Instituto de Matem\'atica \\
Campinas - SP  \\
Brazil - 13083-859

\vspace{1 mm}

\noindent
E-mail: \,{\tt dicesar@ime.unicamp.br} \\

\noindent

Marcus Vin\'icius de Andrade Neves\\
Department of Mathematics, Federal University of Mato Grosso\\
Av. dos Estudantes 5055 \\
78735-901 Rondon\'opolis, MT, Brazil

\vspace{1 mm}

\noindent E-mail: \,{\tt marcusmatematico@hotmail.com}

\end{document}